\documentclass[
12pt, % The default document font size, options: 10pt, 11pt, 12pt
english, % ngerman for German
doublespacing
parskip, % Uncomment to add space between paragraphs
headsepline, % Uncomment to get a line under the header
consistentlayout, % Uncomment to change the layout of the declaration, abstract and acknowledgements pages to match the default layout
]{MastersDoctoralThesis} % The class file specifying the document structure

\usepackage{amsmath}

\usepackage[utf8]{inputenc} % Required for inputting international characters
\usepackage[T1]{fontenc} % Output font encoding for international characters

\usepackage{mathpazo} % Use thThe theory of the structure $\langle\mathbb{N};<,+\rangle$ is decidablee Palatino font by default

\usepackage[autostyle=true]{csquotes} % Required to generate language-dependent quotes in the bibliography

\usepackage{euscript}
\usepackage{oldgerm}
\usepackage{calligra}
\usepackage{amssymb}

\usepackage{graphicx}
\usepackage{hyperref}
\usepackage{mathptmx}

\usepackage{setspace}
\newcommand{\qed}{\hfill$\boxtimes\hspace{-1.725ex}\boxplus$}
\newcommand{\qedef}{\hfill$\otimes\hspace{-1.75ex}\oplus$}
\newcommand{\qecon}{\hfill$\circledast$}
\newtheorem{theorem}{Theorem}[section]
\newtheorem{definition}[theorem]{Definition}%[section]
\newtheorem{pro}[theorem]{Proposition}%[section]
\newtheorem{cor}[theorem]{Corollary}%[section]
\newtheorem{lemma}[theorem]{Lemma}%[section]
\newtheorem{remark}[theorem]{Remark}%[section]
\newtheorem{convention}[theorem]{Convention}%[section]
\newenvironment{proof}{\noindent\mbox{\bf Proof. }}
{}
\thesistitle{Decidability of the Multiplicative\\
and Order Theory of Numbers} % Your thesis title, this is used in the title and abstract, print it elsewhere with \ttitle
\supervisor{Saeed \textsc{Salehi}} % Your supervisor's name, this is used in the title page, print it elsewhere with \supname
\examiner{} % Your examiner's name, this is not currently used anywhere in the template, print it elsewhere with \examname
\degree{\textsc{Doctor of Philosophy}} % Your degree name, this is used in the title page and abstract, print it elsewhere with \degreename
\author{Ziba \textsc{Assadi}} % Your name, this is used in the title page and abstract, print it elsewhere with \authorname
\addresses{} % Your address, this is not currently used anywhere in the template, print it elsewhere with \addressname

\subject{Mathematical Sciences} % Your subject area, this is not currently used anywhere in the template, print it elsewhere with \subjectname
 % Keywords for your thesis, this is not currently used anywhere in the template, print it elsewhere with \keywordnames
\university{University of Tabriz} % Your university's name and URL, this is used in the title page and abstract, print it elsewhere with \univname
\department{Mathematics} % Your department's name and URL, this is used in the title page and abstract, print it elsewhere with \deptname
\group{Pure Mathematics} % Your research group's name and URL, this is used in the title page, print it elsewhere with \groupname
\faculty{Mathematical Sciences} % Your faculty's name and URL, this is used in the title page and abstract, print it elsewhere with \facname

\AtBeginDocument{
\hypersetup{pdftitle=\ttitle} % Set the PDF's title to your title
\hypersetup{pdfauthor=\authorname} % Set the PDF's author to your name
\hypersetup{pdfkeywords=\keywordnames} % Set the PDF's keywords to your keywords
}

\begin{document}

\frontmatter % Use roman page numbering style (i, ii, iii, iv...) for the pre-content pages

\pagestyle{plain} % Default to the plain heading style until the thesis style is called for the body content

%----------------------------------------------------------------------------------------
%	TITLE PAGE
%----------------------------------------------------------------------------------------

\begin{titlepage}
\begin{center}

\includegraphics[scale=0.5]{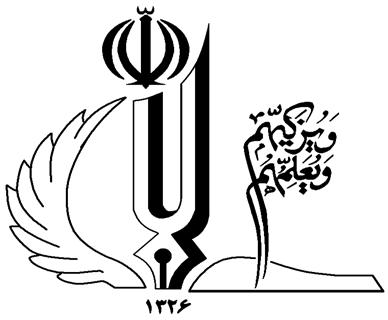}

\vspace{-1.55cm}
\vspace*{.06\textheight}

%{\Large \textfrak{University of Tabriz}}

{\large \textswab{University of Tabriz}}

%{\Large \textgoth{University of Tabriz}}

%{\scshape\large \univname\par}
\vspace{2.5cm} % University name
\textsc{\large Faculty of Mathematical Sciences\\
University of Tabriz, IRAN
\\[1.5cm]} % Thesis type

\HRule \\[0.4cm] % Horizontal line
{\huge \bfseries \ttitle\par}\vspace{0.4cm} % Thesis title
\HRule \\[1.5cm] % Horizontal line

\vspace{0.75cm}

\emph{Author:}\\
{\authorname}

\vspace{0.75cm}

\emph{Supervisor:} \\
{\supname}
%
%\vspace{1.75cm}
%
%\begin{minipage}[t]{0.4\textwidth}
%\begin{flushleft} \large
%\emph{Author:}\\
%{\authorname} % Author name - remove the \href bracket to remove the link
%\end{flushleft}
%\end{minipage}
%\begin{minipage}[t]{0.4\textwidth}
%\begin{flushright} \large
%\emph{Supervisor:} \\
%{\supname} % Supervisor name - remove the \href bracket to remove the link
%\end{flushright}
%\end{minipage}\\[3cm]

\vfill

\large \textit{A Thesis Submitted in Partial Fulfillment of the Requirements\\ for the degree of \degreename } (Ph.D.)\\%[0.3cm]
% University requirement text
\textit{in Pure Mathematics \textup{(}Mathematical Logic\textup{)}}\\[2cm]
%\groupname\\
%\deptname\\[2cm] % Research group name and department name

\vfill

{\large January 2019}\\[4cm] % Date

\vfill
\end{center}
\end{titlepage}
\let\cleardoublepage\clearpage

\dedicatory{\LARGE  \textsl{Dedicated To My Supervisor}, \\ \textup{ Professor {\sc Saeed Salehi}}, \\[3.5cm]  \textit{to whom I owe much more than what I can ever express}.}

%----------------------------------------------------------------------------------------
%	LIST OF CONTENTS/FIGURES/TABLES PAGES
%----------------------------------------------------------------------------------------

\tableofcontents % Prints the main table of contents

%\listoffigures % Prints the list of figures

%\listoftables % Prints the list of tables

%----------------------------------------------------------------------------------------
%	ACKNOWLEDGEMENTS
%----------------------------------------------------------------------------------------

\begin{acknowledgements}
\addchaptertocentry{\acknowledgementname} % Add the acknowledgements to the table of contents

\bigskip

\bigskip

\centerline{\Large \textfrak{In the Name of the Creator of Science, Mathematics and Logic}}

\bigskip

\bigskip

\bigskip

First and foremost, I would like to express my most grateful thanks to my supervisor, to whom this thesis is dedicated wholeheartedly, for teaching  me a lot and taking my hands in the hard moments of wandering in the wonderland  of science and research.

\bigskip

I also thank my advisor Professor Jafarsadegh Eivazloo for studying this thesis and for teaching me.

\bigskip

I thank Professors Mohammad Bagheri and Mohammad Shahriari and Jaber Karimpoor for refereeing the thesis and for their fruitful comments and suggestions.

\bigskip

Last but not the least, I am grateful to my parents for their unending love and to my brother and sister for being there when I needed them most.

\bigskip

\hfill
\textswab{Ziba} $\textgoth{A}\mathfrak{ssadi}$, \textswab{2019}

%\hfill
%\textswab{Ziba Assadi, 2018}

%\hfill
%\textfrak{Ziba Assadi, 2018}

%\hfill
%$\mathfrak{Ziba\ Assadi, 2018}$

\end{acknowledgements}

%----------------------------------------------------------------------------------------
%	ABSTRACT PAGE
%----------------------------------------------------------------------------------------

\begin{abstract}
\addchaptertocentry{\abstractname} % Add the abstract to the table of contents

\noindent
\HRule\\[-0.25cm]

\noindent
The  ordered structures  of natural, integer, rational and real numbers are studied in this thesis. The theories of these numbers in the language of order are decidable and finitely axiomatizable. Also, their theories in the language of order and addition are decidable and infinitely axiomatizable. For the language of order and multiplication, it is known that the theories of $\mathbb{N}$ and $\mathbb{Z}$ are not decidable (and so not axiomatizable by any computably enumerable set of sentences). By Tarski's theorem, the multiplicative ordered structure of $\mathbb{R}$ is decidable also. In this thesis we prove this result directly by quantifier elimination and present an explicit infinite axiomatization. The structure of $\mathbb{Q}$ in the language of order and multiplication seems to be missing in the literature. We show the decidability of its theory by the technique of quantifier elimination and after presenting an infinite axiomatization for this structure, we prove that it is not finitely axiomatizable.
\end{abstract}

\noindent
\HRule

\vspace{1cm}


\mainmatter % Begin numeric (1,2,3...) page numbering

\pagestyle{thesis} % Return the page headers back to the "thesis" style

% Include the chapters of the thesis as separate files from the Chapters folder
% Uncomment the lines as you write the chapters

\doublespacing
% Chapter Template

\chapter*{Introduction} % Main chapter title
\addcontentsline{toc}{chapter}{Introduction}
\label{intro} % Change X to a consecutive number; for referencing this chapter elsewhere, use \ref{ChapterX}

%----------------------------------------------------------------------------------------
%	SECTION 1
%----------------------------------------------------------------------------------------
{\em Entscheidungsproblem\label{dp}}, one of the fundamental problems of (mathematical) logic, asks for a
single-input Boolean-output algorithm that takes a formula $\varphi$ as input and outputs {\tt `yes'} if $\varphi$ is logically valid and outputs {\tt `no'} otherwise.
 Now, we know that this problem is not (computably) solvable.  One reason for this is the existence of an essentially undecidable and finitely axiomatizable theory, see e.g.~\cite{visser}; for another proof see~\cite[Theorem 11.2]{bbj}. However, by G\"odel's completeness theorem, the set of logically valid formulas is computably enumerable, i.e., there exists an input-free algorithms that (after running) lists all the valid formulas (and nothing else).
For the structures, since their theories are complete, the story is different: the theory of a structure is either decidable or that structure is not axiomatizable (by any computably enumerable set of sentences; see e.g.~\cite[Corollaries 25G and 26I]{enderton} or~\cite[Theorem 15.2]{monk}).
%%%%%%%%%%%%%%%
%%%%%%%%%%%%%%%
Axiomatizability or decidability of theories of natural, integer, rational, real and complex numbers in different languages have long been considered by logicians and mathematicians.
%%%%%%%%%%%%%%%
%%%%%%%%%%%%%%%
For example, the additive theory of natural numbers $\langle\mathbb{N};+\rangle$
    was shown to be decidable by Presburger in 1929
     (and by Skolem in 1930; see \cite{smorynski}).
     The multiplicative theory of the natural numbers
     $\langle\mathbb{N};\times\rangle$ was announced  to be decidable by Skolem in 1930. Then it was expected that the theory of addition and multiplication of natural numbers
would be decidable too; confirming Hilbert's Program. But the world was shocked in 1931
 by G\"odel's incompleteness theorem which implies that  the theory of
  $\langle\mathbb{N};+,\times\rangle$ is undecidable (see the subsection~\ref{nom} below).
In this thesis we study the theories of the sets $\mathbb{N}$, $\mathbb{Z}$, $\mathbb{Q}$ and $\mathbb{R}$ in the languages $\{<\}$, $\{<,+\}$ and $\{<,\times\}$; see the table below.

\begin{center}
\begin{tabular}{|c||c|c|c|c|}
\hline
  & $\mathbb{N}$ & $\mathbb{Z}$ & $\mathbb{Q}$ & $\mathbb{R}$ \\
%\hline\\[-3ex]
\hline
\hline
%\hline\\[-2.5ex]
$\{<\}$ & Thm.~\ref{thm-on} & Thm.~\ref{thm-oz} & Thm.~\ref{thm-o} & Thm.~\ref{thm-o} \\
\hline
$\{<,+\}$ & Thm.~\ref{rem-noa} & Thm.~\ref{thm-oaz} & Thm.~\ref{thm-oa} & Thm.~\ref{thm-oa} \\
\hline
$\{<,\times\}$ & Prop.~\ref{subsec-n} & Prop.~\ref{subsec-z} & Cor.~\ref{omq+} & Thm.~\ref{thm-or} \\
%\hline\\[-3ex]
\hline
%\hline\\[-2.5ex]
$\{+,\times\}$ & \cite{enderton} & Prop.~\ref{subsec-z} &   Prop.~\ref{subsec-q} & Subsec.~\ref{subsec-r} \\
\hline
\end{tabular}
\end{center}

Let us note that order is definable in the language $\{+,\times\}$ in these sets: in $\mathbb{N}$ by
  $x<y\iff \exists z (z\!+\!z \neq z\wedge x\!+\!z\!=\!y)$, and in $\mathbb{Z}$ by Lagrange's four square theorem\label{fslt} $x<y$ is equivalent with    $\exists t,u,v,w (x\!\neq\!y \wedge x\!+\!t\!\cdot\!t\!+\!u\!\cdot\!u\!+\!v\!\cdot\!v\!+\!w\!\cdot\!w =y).$ The four square  theorem holds in $\mathbb{Q}$ too: for any   $p/q\!\in\!\mathbb{Q}^+$ we have $pq\!>\!0$ so  $pq\!=\!a^2\!+\!b^2\!+\!c^2\!+\!d^2$ for some integers $a,b,c,d$; therefore, %the equality
  $p/q\!=\!pq/q^2\!=\!(a/q)^2\!+\!(b/q)^2\!+\!(c/q)^2\!+\!(d/q)^2$ holds. Thus, the same formula defines the order ($x<y$) in $\mathbb{Q}$ as well. Finally, in   $\mathbb{R}$  the relation      $x<y$ is equivalent with the formula $\exists z (z\!+\!z\!\neq\!z\wedge x+z\!\cdot\!z=y)$.

The decidability of $\mathbb{N},\mathbb{Z},\mathbb{Q},\mathbb{R}$ in the languages $\{<\}$ and $\{<,+\}$ is already known.  It is also known that the theories of   $\mathbb{N}$ and $\mathbb{Z}$ in the language $\{<,\times\}$ are undecidable,
because the addition operation is definable in the multiplicative ordered structure of natural numbers by Tarski-Robinson's identity. Whence, the theory of $\langle\mathbb{N};\times,<\rangle$ is undecidable. This also holds for the domain of the integer numbers, since the addition operation is definable in $\langle\mathbb{Z};\times,<\rangle$ which implies the undecidability of the theory of $\langle\mathbb{Z};\times,<\rangle$.
%%%%%%%%%%%%%%%%%
%%%%%%%%%%%%%%%%%
The theory of    $\mathbb{R}$ in the language $\{<,\times\}$ is decidable by Tarski-Seidenberg's theorem\label{tst} which states the decidability of the theory of  $\langle\mathbb{R};<,+,\times\rangle$ by showing that $\langle\mathbb{R};<,+,\times\rangle$ is aximatizable by the theory of real closed ordered fields.
%%%%%%%%%%%%%%%%%
%%%%%%%%%%%%%%%%%
Indeed, no heavy algebraic tools are needed for axiomatizing the multiplicative order theory of the real numbers, $\langle\mathbb{R};\times,<\rangle$. The proof of Tarski's theorem appears in a few number of logic books; see e.g.~\cite{az} and~\cite{kk}. Interestingly, the algebraic-geometric proof is more beautiful and more clever; see e.g.~\cite{bcr} and~\cite{bpr}. Although this theorem of Tarski implies the decidability of $\langle\mathbb{R};\times,<\rangle$, it does not present an explicit axiomatization for this structure.
%%%%%%%%%%%%%%%%%
%%%%%%%%%%%%%%%%%
Here, we prove this directly by presenting an explicit axiomatization. Finally, the structure $\langle\mathbb{Q};<,\times\rangle$ is studied in this thesis (seemingly, for the first time). We show, by the method of quantifier elimination, that the theory of this structure is decidable. Here,   the \mbox{(super-)structure}  $\langle\mathbb{Q};+,\times\rangle$ is not usable since it is undecidable (proved by Robinson~\cite{robinson}; see also \cite[Theorem 8.30]{smorynski}). On the other hand its (sub-)structure $\langle\mathbb{Q};\times\rangle$ is decidable (proved in~\cite{mostowski} by Mostowski; see also~\cite{salehi-m}). So, the three structures  $\langle\mathbb{Q};+,\times\rangle$ and $\langle\mathbb{Q};<,\times\rangle$ and $\langle\mathbb{Q};\times\rangle$ are different from each other; the order relation  $<$ is not definable in $\langle\mathbb{Q};\times\rangle$ and the addition operation $+$ is not definable in $\langle\mathbb{Q};<,\times\rangle$ (by our results). 
% Chapter Template

\chapter{Some Preliminaries} % Main chapter title
\label{Chapter1} % Change X to a consecutive number; for referencing this chapter elsewhere, use \ref{ChapterX}

%----------------------------------------------------------------------------------------
%	SECTION 1
%----------------------------------------------------------------------------------------

\section{Ordered Structures}
\begin{definition}[Ordered Structure]\rm\label{def-os}
An {\em ordered structure} is a triple $\langle A;<,\mathcal{L}\rangle$ in which $A$ is a non-empty set and $<$ is a binary relation on $A$ which satisfies the following axioms:
\begin{equation*}
\begin{tabular}{l}
($\texttt{O}_1$) \; $\forall x,y(x<y\rightarrow y\not<x)$,  \\ ($\texttt{O}_2$) \; $\forall x,y,z (x<y<z\rightarrow x<z),$ and \\
($\texttt{O}_3$) \; $\forall x,y (x<y \vee x=y \vee y<x)$;
\end{tabular}
\end{equation*}
\noindent
and $\mathcal{L}$ is a first-order language.
\qedef\end{definition}
Here, $\mathcal{L}$ could be empty, or any language, for example $\{+\}$ or $\{\times\}$ or $\{+,\times\}$.
\section{Various Types of Orders}

\begin{definition}[Dense Linear Order\label{dlo}]\rm
A linear  order relation $<$ is called {\em dense} if it satisfies
\begin{equation*}
\begin{tabular}{l}
($\texttt{O}_4$) \; $\forall x,y (x<y\rightarrow\exists z [x<z<y])$.
\end{tabular}
\end{equation*}
\qedef\end{definition}

\begin{definition}[Orders Without Endpoints\label{weo}]\rm
An order relation $<$ is called
{\em without endpoints} if it satisfies
\begin{equation*}
\begin{tabular}{l}
 ($\texttt{O}_5$) \; $\forall x\exists y (x<y),$  and  \\ ($\texttt{O}_6$) \; $\forall x\exists y (y<x)$.
\end{tabular}
\end{equation*}
\qedef\end{definition}

\begin{definition}[Discrete Order\label{dco}]\rm
A {\em discrete} order has the property that   any element has an immediate successor (i.e., there is no other element in between them). If the successor of $x$ is denoted by $\mathfrak{s}(x)$, then a discrete order satisfies
\begin{equation*}
\begin{tabular}{l}
($\texttt{O}_7$) \; $\forall x,y (x\!<\!y \;\leftrightarrow\; \mathfrak{s}(x)\!<\!y \vee \mathfrak{s}(x)\!=\!y)$.
\end{tabular}
\end{equation*}
\qedef\end{definition}

\begin{convention}\rm
The successor\label{su} of an integer $x$ is $\mathfrak{s}(x)=x+1$.
\qecon\end{convention}

\section{The Main Lemma of Quantifier Elimination}

\begin{definition}[Disjunctive Normal Form\label{djnf}]\rm
The disjunctive normal form of a formula is another formula such that (i) is equivalent to the original formula, and (ii) is the disjunction of some formulas each of wich is the conjunction of some atomic or negated-atomic formulas.
\qedef\end{definition}

\begin{remark}\rm\label{dnf}
Every quantifier-free formula can be written equivalently in disjunctive normal form by elimination of connectives other than $\{\vee, \wedge, \neg\}$ using DeMorgan's laws and the double negation rule, and distributing $\wedge$ over $\vee$, if any.
\qecon\end{remark}

The following lemma which is known as ``{\em The Main Lemma of Quantifier Elimination\label{e}}'',  has been proved in e.g.
\cite[Theorem 31F]{enderton},
\cite[Lemma 2.4.30]{hinman},
\cite[Theorem 1, Chapter 4]{kk},
\cite[Lemma 3.1.5]{marker}
and \cite[Lemma 4.1, Chapter \textrm{III}]{smorynski}.
\begin{lemma}[The Main Lemma of Quantifier Elimination]\rm\label{mainlem}
A theory (or a structure) admits quantifier elimination if and only if every formula of the form $\exists x (\bigwedge\hspace{-1.5ex}\bigwedge_{i}\alpha_i)$ is equivalent with  a quantifier-free formula, where each $\alpha_i$ is an atomic formula or the negation of an atomic formula.
\end{lemma}

\begin{proof}\rm
The ``only if'' part is obvious. We prove the ``if''  part by induction on the complexity of $\varphi$. The statement holds for quantifier-free formulas. So it suffices to check quantifiers: $\forall$ and $\exists$. By the equivalence $\forall x \varphi\equiv\neg\exists\neg\varphi$, the  universal quantifier is reducible to the existential quantifier.
Therefore, the quantifier elimination of the formula $\exists x\varphi$ suffices, where $\varphi$ is quantifier-free. Now, by Convention~\ref{dnf}, every quantifier-free formula can be written in the conjunctive normal form. So we have:
$$\exists x \varphi\equiv\exists x \bigvee\hspace{-2.1ex}\bigvee_{j}(\bigwedge\hspace{-2.1ex}\bigwedge_{i}\alpha_{i,j})
\equiv\bigvee\hspace{-2.1ex}\bigvee_{j}(\exists x (\bigwedge\hspace{-2.1ex}\bigwedge_{i}\alpha_{i,j}))$$
By the assumption, each formula
$\exists x (\bigwedge\hspace{-1.5ex}\bigwedge_{i}\alpha_{i,j})$
is equivalent with  a quantifier-free formula. So, the formula $\exists x\varphi$ is also equivalent with  a quantifier-free formula.
\qed\end{proof}

\begin{remark}\label{hn}\rm
In the presence of a linear order relation ($<$) by the two equivalences $(s\neq t) \leftrightarrow (s<t \vee t<s)$ and  $(s\not < t) \leftrightarrow (t<s\vee t=s)$, which follow from the axioms  $\{\texttt{O}_1,\texttt{O}_2,\texttt{O}_3\}$ (of Definition~\ref{def-os}), we do not need to consider the negated atomic formulas (when there is no   relation symbol other than $<,=$).
\qecon\end{remark} 
% Chapter Template

\chapter{Ordered Structures of Numbers} % Main chapter title
\label{Chapter2} % Change X to a consecutive number; for referencing this chapter elsewhere, use \ref{ChapterX}

%----------------------------------------------------------------------------------------
%	SECTION 1
%----------------------------------------------------------------------------------------

\section{Axiomatizability and Quantifier Elimination}

\begin{definition}[Theory\label{ty}]\rm
A {\em theory} is a set of sentences which is closed under the logical deduction.
\qedef\end{definition}

\begin{definition}[Complete theory\label{cty}]\rm
A theory $T$ is said to be {\em complete} if for every sentence $\sigma$ either
$\sigma\in T$ or $(\neg\sigma)\in T$.
\qedef\end{definition}

\begin{remark}\rm\label{ct}
Since the theory of a structure is a set of sentences which are satisfied within that structure, this theory is complete.
\qecon\end{remark}

\begin{definition}[Decidable set\label{ds}]\rm
A set $A$ of expressions is {\em decidable} if and only if  there exists an effective procedure that,
given an expression $\alpha$, will decide whether or not $\alpha\in A$.\qedef\end{definition}

\begin{definition}[Effectively enumerable set\label{ees}]\rm
A set $A$ of expressions is {\em effectively enumerable} if and only if  there exists an effective procedure that lists, in some order, the members of $A$.
\qedef\end{definition}

\begin{definition}[Axiomatizability\label{ax}]\rm
The theory of a structure $\mathcal{A}=\langle A;\mathcal{L}\rangle$ is {\em axiomatizable} if and only if there exists a decidable set of $\mathcal{L}-$sentences such that the set of its logical consequences is equal to the theory of $\mathcal{A}$.
\qedef\end{definition}
$\bullet$
The structure $\mathcal{A}$ is {\em finitely axiomatizable\label{fa}} if the above set of sentences is finite.

\begin{pro}\label{ctd}\rm
For a finite or countable language:
\begin{itemize}
\item[(1)]
An axiomatizable theory is effectively enumerable.
\item[(2)]
A complete axiomatizable theory is decidable.
\end{itemize}
\end{pro}
\begin{proof}\rm
These results have been proved in e.g. \cite[Corollaries 25F and 25G]{enderton}.
\qed\end{proof}

\begin{remark}\rm
By Remark~\ref{ct} and Proposition~\ref{ctd} the theory of an axiomatizable structure is decidable.
\qedef\end{remark}%\vspace{0.5cm}

\begin{definition}\rm
(The theory of) A structure $\mathcal{A}=\langle A;\mathcal{L}\rangle$ {\em admits quantifier elimination} if and only if every formula in the language $\mathcal{L}$
is equivalent to a quantifier-free formula in the same language with the same free variables.
\qedef\end{definition}
$\bullet$
Since every atom can be proved or disproved, so can the quantifier-free sentences. Whence, the {\em Quantifier Elimination Algorithm} is in fact a {\em Decision Algorithm\label{da}}. %\vspace{0.5cm}

\noindent
$\bullet$
Here, we have presented axiomatizations for structures and have eliminated the quantifiers of their theories. Whence, axiomatizability and decidability of the structures are proved this way.
%-----------------------------------
%	SUBSECTION 1
%-----------------------------------
\subsection{Finite Axiomatizability of $\langle\mathbb{R};<\rangle$ and $\langle\mathbb{Q};<\rangle$}

\begin{convention}\rm
The axioms of The Finite Theory of Dense Linear Orders Without Endpoints are as follows:
\begin{equation*}
\begin{tabular}{l}
($\texttt{O}_1$) \; $\forall x,y(x<y\rightarrow y\not<x)$ \\
($\texttt{O}_2$) \; $\forall x,y,z (x<y<z\rightarrow x<z)$ \\
($\texttt{O}_3$) \; $\forall x,y (x<y \vee x=y \vee y<x)$ \\
($\texttt{O}_4$) \; $\forall x,y (x<y\rightarrow\exists z [x<z<y])$\\
($\texttt{O}_5$) \; $\forall x\exists y (x<y)$\\
($\texttt{O}_6$) \; $\forall x\exists y (y<x)$
\end{tabular}
\end{equation*}
\qecon\end{convention}

The following theorem has been proved in
\cite[Theorems 2.4.1 and 3.1.3]{marker}.

$\bullet$
 Here, we present a syntactic (proof-theoretic) proof.

 \begin{theorem}\rm\label{thm-o}
The  finite theory of dense linear orders without endpoints \textup{(}with the axioms  $\{\texttt{O}_1,\texttt{O}_2,\texttt{O}_3,\texttt{O}_4,\texttt{O}_5,\texttt{O}_6\}$\textup{)}
completely axiomatizes the order theory of the   real and rational  numbers and, moreover, the structures $\langle\mathbb{R};<\rangle$ and $\langle\mathbb{Q};<\rangle$ admit quantifier elimination, and so their theories are decidable.
\end{theorem}
\begin{proof}\rm
By Remark~\ref{hn}, all the atomic formulas are either of the form $u<v$ or $u=v$ for some variables $u$ and $v$. If both of the variables are equal then $u<u$ is equivalent with $\bot$  by $\texttt{O}_1$ and $u=u$ is equivalent with $\top$.
So, by Lemma~\ref{mainlem},  it suffices to eliminate the quantifier of the formulas of the form
\begin{equation}\label{f-1}
\exists x (\bigwedge\hspace{-2.15ex}\bigwedge_{i<\ell} y_i<x \wedge \bigwedge\hspace{-2.55ex}\bigwedge_{j<m} x<z_j \wedge \bigwedge\hspace{-2.35ex}\bigwedge_{k<n} x=u_k)
\end{equation}
where $y_i$'s, $z_j$'s and $u_k$'s are variables. \\
Now, if $n\neq 0$ then the formula~\eqref{f-1} is equivalent with the quantifier-free formula
\begin{equation*}
\bigwedge\hspace{-2.15ex}\bigwedge_{i<\ell} y_i<u_0 \wedge \bigwedge\hspace{-2.55ex}\bigwedge_{j<m} u_0<z_j \wedge \bigwedge\hspace{-2.35ex}\bigwedge_{k<n} u_0=u_k.
\end{equation*}
So, let us suppose that $n=0$. Then if $\ell=0$ or $m=0$, the formula~\eqref{f-1} is equivalent with the quantifier-free formula $\top$, by the axioms   $\texttt{O}_5$ and $\texttt{O}_6$  (with $\texttt{O}_2$ and $\texttt{O}_3$) respectively,  and if $\ell,m\neq 0$, it is equivalent with the quantifier-free formula $\bigwedge\hspace{-1.55ex}\bigwedge_{i<\ell,j<m} y_i<z_j$ by the axiom  $\texttt{O}_4$ (with $\texttt{O}_2$ and $\texttt{O}_3$).
\qed\end{proof}

\begin{cor}\rm
In fact, for any set $A$ such that $\mathbb{Q}\subseteq A\subseteq\mathbb{R}$, the structure $\langle A;<\rangle$ can be completely axiomatized by the finite set of axioms $\{\texttt{O}_1,\texttt{O}_2,\texttt{O}_3,\texttt{O}_4,
\texttt{O}_5,\texttt{O}_6\}$.
\qed\end{cor}
%-----------------------------------
%	SUBSECTION 1
%-----------------------------------

\subsection{Finite Axiomatizability of $\langle\mathbb{Z};<\rangle$}
\begin{pro}\rm
The theory of the structure $\langle\mathbb{Z};<\rangle$ does not admit quantifier elimination.
\end{pro}
\begin{proof}\rm
We show that the formula $\exists x (y<x<z)$ is not equivalent with any quantifier-free formula in the language $\{<\}$ (note that it is not equivalent with $y<z$): all the atomic formulas with the free variables $y$ and $z$ are $y<z$, $z<y$, $y=y(\equiv\top)$, $z=z(\equiv\top)$, $y<y(\equiv\bot)$ and $z<z(\equiv\bot)$. None of the propositional compositions of these formulas can be equivalent to  the formula $\exists x (y<x<z)$.
\qed\end{proof}

\begin{remark}\rm
If we add the successor operation $\mathfrak{s}$ to the language, we will have:
$$\exists x (y<x<z)\iff\mathfrak{s}(y)<z,$$
and we will show that the process of quantifier elimination will go through in this language [Theorem~\ref{thm-oz}].
\qecon\end{remark}

\begin{convention}\rm
The axioms of The  Finite  Theory of Discrete  Linear Orders Without Endpoints are as follows:
\begin{equation*}
\begin{tabular}{l}
($\texttt{O}_1$) \; $\forall x,y(x<y\rightarrow y\not<x)$ \\
($\texttt{O}_2$) \; $\forall x,y,z (x<y<z\rightarrow x<z)$ \\
($\texttt{O}_3$) \; $\forall x,y (x<y \vee x=y \vee y<x)$ \\
($\texttt{O}_7$) \; $\forall x,y (x\!<\!y \;\leftrightarrow\; \mathfrak{s}(x)\!<\!y \vee \mathfrak{s}(x)\!=\!y)$\\
($\texttt{O}_8$) \; $\forall x\exists y (\mathfrak{s}(y)=x)$
\end{tabular}
\end{equation*}
\qecon\end{convention}
$\bullet$
The following has been proved earlier; see ~\cite[Theorem 2.12]{rz}.
\begin{theorem}\rm\label{thm-oz}
The  finite  theory of discrete  linear orders without endpoints, consisting of the axioms $\{\texttt{O}_1,\texttt{O}_2,\texttt{O}_3,\texttt{O}_7,\texttt{O}_8\}$, completely axiomatizes the order theory of the   integer   numbers and, moreover, the structure $\langle\mathbb{Z};<,\mathfrak{s}\rangle$ admits quantifier elimination, and so its theory is decidable.
\end{theorem}
\begin{proof}\rm
We  note that all the terms in the language $\{<,\mathfrak{s}\}$ are of the form $\mathfrak{s}^n(y)$ for some variable $y$ and $n\in\mathbb{N}$. So, by Remark~\ref{hn}, all the atomic formulas are of the form    $\mathfrak{s}^n(u)=\mathfrak{s}^m(v)$ or $\mathfrak{s}^n(u)<\mathfrak{s}^m(v)$, for some variables $u,v$. If a variable $x$ appears in the both sides of an atomic formula, then we have either $\mathfrak{s}^n(x)=\mathfrak{s}^m(x)$ or $\mathfrak{s}^n(x)<\mathfrak{s}^m(x)$. The formula $\mathfrak{s}^n(x)=\mathfrak{s}^m(x)$ is equivalent with $\top$ when $n=m$ and with $\bot$ otherwise; also $\mathfrak{s}^n(x)<\mathfrak{s}^m(x)$ is equivalent with $\top$ when $n<m$ and with $\bot$ otherwise. So, it suffices to consider the atomic formulas of the form  $t<\mathfrak{s}^n(x)$ or $\mathfrak{s}^n(x)<t$ or $\mathfrak{s}^n(x)=t$, for some $x$-free term $t$ and $n\in\mathbb{N}^+$. Now, by Lemma~\ref{mainlem}, we eliminate the quantifier of the following formulas
\begin{equation}\label{f-2}
\exists x (\bigwedge\hspace{-2.15ex}\bigwedge_{i<\ell} t_i<\mathfrak{s}^{p_i}(x) \wedge \bigwedge\hspace{-2.55ex}\bigwedge_{j<m} \mathfrak{s}^{q_j}(x) <s_j \wedge \bigwedge\hspace{-2.35ex}\bigwedge_{k<n} \mathfrak{s}^{r_k}(x)=u_k).
\end{equation}
The axiom $\texttt{O}_7$ proves $[a<b] \leftrightarrow [\mathfrak{s}(a)<\mathfrak{s}(b)]$ and $[a=b] \leftrightarrow [\mathfrak{s}(a)=\mathfrak{s}(b)]$; so
we can assume that $p_i$'s and $q_j$'s and $r_k$'s in
the formula~\eqref{f-2} are equal to each other, say to $\alpha$. Then, by $\texttt{O}_8$, the formula~\eqref{f-2} is equivalent with
\begin{equation}\label{f-3}
\exists y (\bigwedge\hspace{-2.15ex}\bigwedge_{i<\ell} t_i'<y \wedge \bigwedge\hspace{-2.55ex}\bigwedge_{j<m} y<s_j' \wedge \bigwedge\hspace{-2.35ex}\bigwedge_{k<n} y=u_k'),
\end{equation}
for some (possibly new)  terms $t_i',s_j',u_k'$ (and   $y=\mathfrak{s}^\alpha(x)$).\\
Now, if $n\neq 0$, then the formula \eqref{f-3} is equivalent with the quantifier-free formula
\begin{equation*}
\bigwedge\hspace{-2.15ex}\bigwedge_{i<\ell} t_i'<u_0' \wedge \bigwedge\hspace{-2.55ex}\bigwedge_{j<m} u_0'<s_j' \wedge \bigwedge\hspace{-2.35ex}\bigwedge_{k<n} u_0'=u_k'.
\end{equation*}
Let us then assume that $n=0$. The formula
\begin{equation}\label{f-4}
\exists x (\bigwedge\hspace{-2.15ex}\bigwedge_{i<\ell} t_i<x \wedge \bigwedge\hspace{-2.55ex}\bigwedge_{j<m} x<s_j)
\end{equation}
is equivalent with the quantifier-free formula
$$\bigwedge\hspace{-2ex}\bigwedge_{i,j} \mathfrak{s}(t_i)<s_j$$
by the axiom $\texttt{O}_7$.
\qed\end{proof}
%----------------------------------------------------------------------------------------
%	SECTION 1
%----------------------------------------------------------------------------------------

\subsection{Finite Axiomatizability of $\langle\mathbb{N};<\rangle$}

\begin{pro}\rm
The theory of the structure $\langle\mathbb{N};<\rangle$ does not admit quantifier elimination.
\end{pro}
\begin{proof}\rm
We show that the formula $\exists x (\mathfrak{s}(x)=y)$ is not equivalent with any quantifier-free formula. All the atomic formulas with the free variable $y$ are either of the form $y<y$ or $y=y$. The equivalences $(y<y)\equiv\bot$  and $(y=y)\equiv\top$ show that none of the propositional compositions of them can be equivalent to   $\exists x (\mathfrak{s}(x)=y)$, because its truth depends on $y$ \textup{(}it is equivalent with $\bot$ for $y=0$ and with $\top$ otherwise\textup{)}.
\qed\end{proof}

\begin{remark}\rm
By adding the constant $\bf 0$ to the language $\{<\}$ we will have:
$$\exists x (x<y)\iff\texttt {\bf 0}<y.$$
Still quantifier elimination is not possible [Proposition~\ref{hs1}, below].
\qecon\end{remark}

\begin{pro}\rm\label{hs1}
The theory of the structure $\langle\mathbb{N};<,\bf 0\rangle$ does not admit quantifier elimination.
\end{pro}
\begin{proof}\rm
It suffices to show that the formula $\exists x (y<x<z)$ is not equivalent with any quantifier-free formula. All the atomic formulas with the free variables $y$ and $z$ are $y=0$, $z=0$, $0<y$, $0<z$, $y=y(\equiv\top)$, $z=z(\equiv\top)$, $y<y(\equiv\bot)$, $z<z(\equiv\bot)$, $y=z$, $z=y$, $z<y$ and $y<z$. None of the propositional compositions of these formulas can be equivalent with  the formula $\exists x (y<x<z)$.
\qed\end{proof}

\begin{remark}\rm
If we add the successor operation $\mathfrak{s}$ to the language $\{<\}$ we will have:
$$\exists x (y<x<z)\iff\mathfrak{s}(y)<z,$$
and now we show that the quantifier elimination is still not possible in the language $\{<,\mathfrak{s}\}$ [Proposition~\ref{hs2}, below].
\qecon\end{remark}

\begin{pro}\rm\label{hs2}
The theory of the structure $\langle\mathbb{N};<,\mathfrak{s}\rangle$ does not admit quantifier elimination.
\end{pro}
\begin{proof}\rm
We show that the formula $\exists x (\mathfrak{s}(x)=y)$ is not equivalent with any quantifier-free formula. All the atomic formulas with the free variable $y$ are either of the form $\mathfrak{s}^n(y)<\mathfrak{s}^m(y)$ or $\mathfrak{s}^n(y)=\mathfrak{s}^m(y)$ which do not depend on $y$ and are equivalent to either $\top$ or $\bot$. So, the formula $\exists x (\mathfrak{s}(x)=y)$ \textup{(}which is equivalent with $\bot$ for $y=0$ and with $\top$ otherwise\textup{)} is not equivalent with any quantifier-free $\{<,\mathfrak{s}\}$-formula.
\qed\end{proof}

\noindent
In the following we will show the quantifier elimination of the theory of the structure $\langle\mathbb{N};<,\mathfrak{s},{\bf 0}\rangle$. This theorem has been proved in \cite[Theorem~32A]{enderton}. \bigskip
\begin{theorem}\rm\label{thm-on}
The following axioms completely axiomatize the order theory of the ordered natural numbers:
\begin{equation*}
\begin{tabular}{l}
($\texttt{O}_1$) \; $\forall x,y(x<y\rightarrow y\not<x)$ \\
($\texttt{O}_2$) \; $\forall x,y,z (x<y<z\rightarrow x<z)$ \\
($\texttt{O}_3$) \; $\forall x,y (x<y \vee x=y \vee y<x)$ \\
($\texttt{O}_7$) \; $\forall x,y (x\!<\!y \;\leftrightarrow\; \mathfrak{s}(x)\!<\!y \vee \mathfrak{s}(x)\!=\!y)$\\
($\texttt{O}_8^\circ$) \; $\forall x\exists y (x\neq {\bf 0} \rightarrow \mathfrak{s}(y)=x)$ \\
($\texttt{O}_9$) \; $\forall x (x\not<{\bf 0})$
\end{tabular}
\end{equation*}
and, moreover, the structure $\langle\mathbb{N};<,\mathfrak{s},{\bf 0}\rangle$ admits quantifier elimination, and so its theory is decidable.
\end{theorem}
\begin{proof}\rm
All the atomic formulas of the free variable $u$ in the language $\{<,\mathfrak{s},\bf 0\}$ are   of the form    $\mathfrak{s}^n(u)=\mathfrak{s}^m(u)$ or $\mathfrak{s}^n(u)<\mathfrak{s}^m(u)$ or
$\mathfrak{s}^n(\texttt 0)=\mathfrak{s}^m(u)$ or $\mathfrak{s}^n(\texttt 0)<\mathfrak{s}^m(u)$ or
$\mathfrak{s}^n(u)<\mathfrak{s}^m(\texttt 0)$. The formula $\mathfrak{s}^n(u)=\mathfrak{s}^m(u)$ is equivalent with $\top$ when $n=m$ and with $\bot$ otherwise; also $\mathfrak{s}^n(u)<\mathfrak{s}^m(u)$ is equivalent with $\top$ when $n<m$ and with $\bot$ otherwise. So, it suffices to consider the atomic formulas of the form  $t<\mathfrak{s}^n(x)$ or $\mathfrak{s}^n(x)<t$ or $\mathfrak{s}^n(x)=t$ for some $x$-free term $t$ and $n\in\mathbb{N}^+$. Now, by Lemma~\ref{mainlem} and  the presence of $<$, which eliminates the negation already, we eliminate the quantifier of the following formulas
\begin{equation}\label{ff-2}
\exists x (\bigwedge\hspace{-2.15ex}\bigwedge_{i<\ell} t_i<\mathfrak{s}^{p_i}(x) \wedge \bigwedge\hspace{-2.55ex}\bigwedge_{j<m} \mathfrak{s}^{q_j}(x) <s_j \wedge \bigwedge\hspace{-2.35ex}\bigwedge_{k<n} \mathfrak{s}^{r_k}(x)=u_k).
\end{equation}
By the provable formulas
$$\mathfrak{s}(x)<\mathfrak{s}(y)\Leftrightarrow x<y\hspace{0.5cm}\text{and}\hspace{0.5cm}\mathfrak{s}(x)=\mathfrak{s}(y)\Leftrightarrow x=y,$$
the formula~\eqref{ff-2}, for $N=\max\{p_i, q_j, r_k\}$, is equivalent with
\begin{equation}\label{ff-3}
\exists x \Big(\bigwedge\hspace{-2.15ex}\bigwedge_{i<\ell}\mathfrak{s}^{N-p_i}(t_i)<\mathfrak{s}^{N}(x) \wedge \bigwedge\hspace{-2.55ex}\bigwedge_{j<m} \mathfrak{s}^{N}(x) <\mathfrak{s}^{N-q_j}(s_j) \wedge \bigwedge\hspace{-2.35ex}\bigwedge_{k<n} \mathfrak{s}^{N}(x)=\mathfrak{s}^{N-r_k}(u_k)\Big).
\end{equation}
Now for $y=\mathfrak{s}^{N}(x)$, $t_i'=\mathfrak{s}^{N-p_i}(t_i)$, $s_j'=\mathfrak{s}^{N-q_j}(s_j)$ and $u_k'=\mathfrak{s}^{N-r_k}(u_k)$ the formula~\eqref{ff-3} is equivalent with
\begin{equation*}\label{ff-4}
\exists y (\bigwedge\hspace{-2.15ex}\bigwedge_{i<\ell} t_i'<y \wedge \bigwedge\hspace{-2.55ex}\bigwedge_{j<m} y <s_j' \wedge \bigwedge\hspace{-2.35ex}\bigwedge_{k<n} y=u_k'\ \ \wedge\ \ \mathfrak{s}^{N}(\texttt 0)\leqslant y).
\end{equation*}
So, it suffices to eliminate the quantifiers of the following formulas:
\begin{equation}\label{ff-4}
\exists y (\bigwedge\hspace{-2.15ex}\bigwedge_{i<\ell} t_i<y \wedge \bigwedge\hspace{-2.55ex}\bigwedge_{j<m} y <s_j \wedge \bigwedge\hspace{-2.35ex}\bigwedge_{k<n} y=u_k).
\end{equation}
If $n\neq 0$, then the formula~\eqref{ff-4} is equivalent with the following quantifier-free formula:
\begin{equation*}
\bigwedge\hspace{-2.15ex}\bigwedge_{i<\ell} t_i<u_0 \wedge \bigwedge\hspace{-2.55ex}\bigwedge_{j<m} u_0<s_j \wedge \bigwedge\hspace{-2.35ex}\bigwedge_{k<n} u_0=u_k.
\end{equation*}
And, if $n=0$, then we eliminate the quantifier of:
\begin{equation}\label{ff-5}
\exists y (\bigwedge\hspace{-2.15ex}\bigwedge_{i<\ell} t_i<y \wedge \bigwedge\hspace{-2.55ex}\bigwedge_{j<m} y<s_j).
\end{equation}
Now, If $\ell=0$, then the formula~\eqref{ff-5} is equivalent with the following quantifier-free formula:
\begin{equation*}
\bigwedge\hspace{-2.55ex}\bigwedge_{j<m} {\bf 0}<s_j.
\end{equation*}
If $m=0$, then the formula~\eqref{ff-5} is equivalent with $\top$.\\
Finally, if $\ell\neq{\bf 0}\neq m$, then the formula~\eqref{ff-5} is equivalent with the following quantifier-free formula:
\begin{equation*}
\bigwedge\hspace{-2.1ex}\bigwedge_{i,j} \mathfrak{s}(t_i)<s_j.
\end{equation*}
\qed\end{proof} 
% Chapter Template

\chapter{Additive Ordered Structures} % Main chapter title
\label{Chapter3} % Change X to a consecutive number; for referencing this chapter elsewhere, use \ref{ChapterX}
In this chapter, we study the structures of the sets $\mathbb{N},\mathbb{Z},\mathbb{Q},\mathbb{R}$ over the language \mbox{$\{+,<\}$.}
%----------------------------------------------------------------------------------------
%	SECTION 1
%----------------------------------------------------------------------------------------

\section{Some Group Theory}

\begin{definition}[Group\label{g}]\rm
A {\em group} is a structure  $\langle G;\ast,{\sf e},\iota\rangle$, where $\ast$ is a binary operation on $G$, ${\sf e}$ is a constant (a special element of $G$) and $\iota$ is a unary operation on $G$, which satisfy the following axioms:
\begin{equation*}
\begin{tabular}{l}
 $\forall x,y,z\,[x\ast (y\ast z)=(x\ast y)\ast z]$; \\
 $\forall x (x\ast {\sf e}=x)$; \\
 $\forall x (x\ast  \iota(x)={\sf e})$.
\end{tabular}
\end{equation*}
\qedef\end{definition}

\noindent
$\bullet$
A group is called {\em non-trivial\label{ntg}} when
\begin{equation*}
\begin{tabular}{l}
 $\exists x (x\neq {\sf e})$.
\end{tabular}
\end{equation*}

\begin{definition}[Abelian group\label{ag}]\rm
A group is called {\em abelian} when it satisfies the commutativity axiom:\\[-2.6em]
$$\forall x,y (x\ast y=y\ast x).$$
\qedef\end{definition}

\begin{definition}[Divisible group\label{dvg}]\rm
A group is called {\em divisible} when for any $n\in\mathbb{N^{+}}$ we have
\begin{equation*}
\begin{tabular}{l}
 $\forall x\exists y[x=\underbrace{y\ast\cdots\ast y}_{\text{n-times}}]$.
\end{tabular}
\end{equation*}
\qedef\end{definition}

\begin{definition}[Ordered group\label{og}]\rm
An {\em ordered group} is a group equipped with an order relation $<$ (which satisfies $\texttt{O}_1,\texttt{O}_2,\texttt{O}_3$) such that also the axiom

\begin{equation*}
\begin{tabular}{l}
$\forall x,y,z(x\!<\!y \;\rightarrow\; x\ast z\!<\!y\ast z \;\wedge\; z\ast x\!<\!z\ast y)$
\end{tabular}
\end{equation*}
is satisfied in it.
\qedef\end{definition}

\begin{remark}\rm
The axioms of The  Theory of Non-trivial Ordered Divisible Abelian Groups in the language $\mathcal{L}=\{<,+,-,0\}$ are as follows:
\begin{equation*}
\begin{tabular}{l}
($\texttt{O}_1$) \; $\forall x,y(x<y\rightarrow y\not<x)$  \\ ($\texttt{O}_2$) \; $\forall x,y,z (x<y<z\rightarrow x<z)$ \\
($\texttt{O}_3$) \; $\forall x,y (x<y \vee x=y \vee y<x)$ \\
($\texttt{A}_1$) \; $\forall x,y,z\,(x+(y+z)=(x+y)+z)$ \\
($\texttt{A}_2$) \; $\forall x (x+\mathbf{0}=x)$ \\
($\texttt{A}_3$) \; $\forall x (x+ (-x)=\mathbf{0})$ \\
($\texttt{A}_4$) \; $\forall x,y (x+y=y+x)$ \\
($\texttt{A}_5$) \; $\forall x,y,z(x<y\rightarrow x+z<y+z)$ \\
($\texttt{A}_6$) \; $\exists y (y\neq {\bf 0})$ \\
($\texttt{A}_7$) \; $\forall x\exists y(x=n\centerdot y)$  \qquad \qquad \qquad  $n\in\mathbb{N}^+$  \\
\end{tabular}
\end{equation*}
\qecon\end{remark}
%-----------------------------------
%	SUBSECTION 2
%-----------------------------------

\section{The Rational and Real Numbers with Order and Addition}

\subsection{Quantifier Elimination of $\langle\mathbb{R};<,+\rangle$ and $\langle\mathbb{Q};<,+\rangle$}

\begin{theorem}\rm\label{thm-oa}
The infinite theory of non-trivial ordered divisible abelian groups completely axiomatizes the order and additive theory of the   real and rational  numbers and, moreover, the structures $\langle\mathbb{R};<,+,-,{\bf 0}\rangle$ and  $\langle\mathbb{Q};<,+,-,{\bf 0}\rangle$ admit quantifier elimination, and so their theories are decidable~\cite[Corollary 3.1.17]{marker}.
\end{theorem}
\begin{proof}
Firstly, let us note that  $\texttt{O}_4$, $\texttt{O}_5$ and $\texttt{O}_6$ can be proved from the presented axioms: if $a<b$ then by $\texttt{A}_7$ there exists some $c$ such that $c+c=a+b$;  one can easily show that $a<c<b$ holds.  Thus $\texttt{O}_4$ is proved; for  $\texttt{O}_5$ note that for any ${\bf 0}<a$ we have $a<a+a$ by $\texttt{A}_5$. A dual argument can prove the axiom $\texttt{O}_6$. Also,  the     equivalences
\begin{itemize}\itemindent=5ex
\item[(i)] $[a<b] \leftrightarrow [n\centerdot a< n\centerdot b]$ and
\item[(ii)] $[a=b] \leftrightarrow [n\centerdot a=n\centerdot b]$
\end{itemize}
can be proved from the axioms: (i) follows from $\texttt{A}_5$ (with $\texttt{O}_1,\texttt{O}_2,\texttt{O}_3$) and (ii) follows from $\forall x (n\centerdot x={\bf 0}\rightarrow x={\bf 0})$ which is derived from $\texttt{A}_5$ (with $\texttt{O}_1,\texttt{O}_2,\texttt{O}_3$).

Secondly, every term containing $x$ is equal to $n\centerdot x + t$ for some $x$-free term $t$ and $n\!\in\!\mathbb{Z}\!-\!\{0\}$. So, every atomic formula containing $x$ is equivalent with $n\centerdot x \Box t$ where $\Box\!\in\!\{=,<,>\}$. Whence, by Remark~\ref{mainlem}, it suffices to prove the equivalence of the formula
\begin{equation}\label{f-5}
\exists x (\bigwedge\hspace{-2.15ex}\bigwedge_{i<\ell} t_i<p_i\centerdot x \wedge \bigwedge\hspace{-2.55ex}\bigwedge_{j<m} q_j\centerdot x <s_j \wedge \bigwedge\hspace{-2.35ex}\bigwedge_{k<n} r_k\centerdot x=u_k)
\end{equation}
with a quantifier-free formula.
  By the equivalences (i) and (ii) above,
we can assume that $p_i$'s and $q_j$'s and $r_k$'s in
the formula~\eqref{f-5} are equal to each other, say to $\alpha$. Then by $\texttt{A}_7$, the formula~\eqref{f-5} is equivalent with
\begin{equation}\label{f-6}
\exists y (\bigwedge\hspace{-2.15ex}\bigwedge_{i<\ell} t_i'<y \wedge \bigwedge\hspace{-2.55ex}\bigwedge_{j<m} y<s_j' \wedge \bigwedge\hspace{-2.35ex}\bigwedge_{k<n} y=u_k')
\end{equation}
for some (possibly new) terms $t_i',s_j',u_k'$ (and   $y=\alpha\centerdot x$).\\
Now, if $n\neq 0$ then the formula~\eqref{f-6} is equivalent with the quantifier-free formula
\begin{equation*}
\bigwedge\hspace{-2.15ex}\bigwedge_{i<\ell} t_i'<u_0 \wedge \bigwedge\hspace{-2.55ex}\bigwedge_{j<m} u_0<s_j' \wedge \bigwedge\hspace{-2.35ex}\bigwedge_{k<n} u_0=u_k'.
\end{equation*}
So, let us suppose that $n=0$. Then if $\ell=0$ or $m=0$, the formula~\eqref{f-6} is equivalent with the quantifier-free formula $\top$, by the axioms   $\texttt{O}_5$ and $\texttt{O}_6$  (with $\texttt{O}_2$ and $\texttt{O}_3$) respectively,  and if $\ell,m\neq 0$, it is equivalent with the quantifier-free formula $\bigwedge\hspace{-1.55ex}\bigwedge_{i<\ell,j<m} t_i'<s_j'$ by the axiom  $\texttt{O}_4$ (with $\texttt{O}_2$ and $\texttt{O}_3$). \textup{(}Compare with the proof of Theorem~\ref{thm-o}\textup{)}
\qed\end{proof}

\subsection{Non-finite Axiomatizability of $\langle\mathbb{R};<,+\rangle$ and $\langle\mathbb{Q};<,+\rangle$}

\begin{pro}\label{rem-infq}\rm
The structures  $\langle\mathbb{R};<,+\rangle$ and $\langle\mathbb{Q};<,+\rangle$ are not finitely axiomatizable.
\end{pro}
\begin{proof}
It suffices to note that for a given natural number $N$, the set
$$\mathbb{Q}/N!=\{{m}/{(N!)^k}\mid m\in\mathbb{Z},k\in\mathbb{N}\}$$
of rational numbers, where $N!=2\times 3\times \cdots \times N$,  is closed under addition  and so satisfies the axioms $\texttt{O}_1$, $\texttt{O}_2$, $\texttt{O}_3$, $\texttt{A}_1$, $\texttt{A}_2$, $\texttt{A}_3$, $\texttt{A}_4$, $\texttt{A}_5$,  $\texttt{A}_6$ and the finite number of the instances of the axiom $\texttt{A}_7$ (for $n=1,\cdots,N$) but does not satisfy the instance of $\texttt{A}_7$ for $n={p}$, where ${p}$ is a prime number larger than $N!$.
\qed\end{proof}

%----------------------------------------------------------------------------------------
%	SECTION 3
%----------------------------------------------------------------------------------------

\section{The Chinese Remainders}
%The {\em Chinese Remainder Theorem} appeared in the ancient.
For eliminating the quantifiers of the formulas of the structure  $\langle\mathbb{Z};<,+\rangle$, we add the (binary) congruence relations $\{\equiv_{n}\}_{n\geqslant 2}$ (modulo standard natural numbers) to the language; let us note that $a\equiv_n b$ is equivalent with $\exists x (a+n\centerdot x = b)$.
About these congruence relations the following Generalized Chinese Remainder Theorem will be useful later.\\
The Chinese Remainder Theorem has been an important tool in astronomical calculations and in religious
observance (what day does Easter fall on?); it has been a source for mathematical puzzles. It
has been abstracted in algebra to a theorem on the isomorphism of one homomorphic image
of a ring of a given type to a product of two homomorphic images of the ring; it has been applied by computer scientists to obtain multiple precision,
and, somewhere along the way, it has been used in logic as a means of coding finite
sequences \cite{smorynski}.
\bigskip

\subsection{The B\'{e}zout's Theorem}\label{bz}

\begin{lemma}\rm[B\'{e}zout's Identity]
Given integers $a$ and $b$, not both of wich are zero, and for $d$ which is the greatest common divisor of $a$ and $b$, there exist integers $x$ and $y$ such that
$$d=ax+by.$$
\end{lemma}
\begin{proof}
Consider the set $S$ of all the positive linear combinations of $a$ and $b$:
$$S=\{au+bv\ |\ u, v\in\mathbb{Z},\ \ au+bv>0\}.$$
Notice first that $S$ is not empty. For example, if $a\neq 0$, then the integer $ |a |=au+b.0$ lies in $S$, where we choose $u=1$ or $u=-1$ according as $a$ is positive or negative. By virtue of the Well-Ordering Principle, $S$ must contain a smallest element $d$. Thus, from the very definition of $S$, there exist integers $x$ and $y$ for which $d=ax+by$ holds. We claim that $d$ is the greatest common divisor of $a$ and $b$.

By the Division Algorithm, we can obtain integers $q$ and $r$ such that $a=qd+r$, where $0\leq r<d$. Then $r$ can be written in the form
\begin{center}
\begin{tabular}{lcl}
$r=a-qd$&$=$&$a-q(ax+by)$\\&$=$&$a(1-qx)+b(-qy)$
\end{tabular}
\end{center}
If $r$ were positive, then this representation would imply that $r$ is a member of $S$, contradicting the fact that $d$ is the least integer in $S$ (recall that $r<d$). Therefore, $r=0$, and so $a=qd$, or equivalently $d | a$. By similar reasoning, $d | b$, the effect of which is to make $d$ a common divisor of $a$ and $b$.

Now if $c$ is an arbitrary positive common divisor of the integers $a$ and $b$, then we conclude that $c | (ax+by)$; that is, $c | d$ and $c= | c |\leq | d | = d$, so that $d$ is greater than every positive common divisor of $a$ and $b$. Piecing the bits of information together, we see that $d$ is the greatest common divisor of $a$ and $b$.

\qed\end{proof}

\subsection{The Chinese Remainder Theorem\label{cr}}

\begin{pro}\rm[Chinese Remainder]\label{cr}
For  integers $n_0,n_1,\cdots,n_k\geqslant 2$ which are pairwise co-prime and arbitrary $t_0,t_1,\cdots,t_k$, there exists some integer $x$ such that  $x\equiv_{n_i} t_i$ for $i=0,\cdots,k$.
\end{pro}
\begin{proof}
We take $m=n_0 n_1\cdots n_k$. Since the integers $n_0,n_1,\cdots,n_k\geqslant 2$ are pairwise co-prime, we have:
\begin{equation}\label{chr-01}
\begin{cases}
(n_0,\frac{m}{n_0})=1 \\
(n_1,\frac{m}{n_1})=1 \\
\ \ \ \ \ \vdots\\
(n_k,\frac{m}{n_k})=1
\end{cases}\
\end{equation}
By lemma~\ref{bz} and relation~\eqref{chr-01}, there exist integers $c_0,c_1,\cdots,c_k$ and $d_0,d_1,\cdots,d_k$ such that:
\begin{equation}\label{chr-02}
\begin{cases}
c_0 n_0+d_0\frac{m}{n_0}=1 \\
c_1 n_1+d_1\frac{m}{n_1}=1 \\
\ \ \ \ \ \ \ \ \vdots\\
c_k n_k+d_k\frac{m}{n_k}=1
\end{cases}
\end{equation}
We show that
$$x=\sum_{i=0}^{k} d_i t_i\frac{m}{n_i}$$
satisfies the conclusion of the theorem. \\
For $j=0,\cdots,k$ we have:
$$x=d_j t_j\frac{m}{n_j}+\sum_{i\neq j} d_i t_i\frac{m}{n_i}$$
by~\eqref{chr-02}
$$\ \ \ \ \ \ \ \ \ \ = t_j(1-c_j n_j)+\sum_{i\neq j} d_i t_i\frac{m}{n_i}$$
$$\ \ \ \ \ \ \ \ \ \ \ \ \ \ \ = t_j+n_j(-t_jc_j+\sum_{i\neq j} d_i t_i\frac{m}{n_j n_i})$$
So, $\ x\equiv_{n_j} t_j$ holds for $j=0,\cdots,k$.
\qed\end{proof}

\subsection{The Generalized Chinese Remainder Theorem\label{crt}}

\begin{lemma}\label{max-min}\rm
For integers $n_0,n_1,\cdots,n_{k+1}$ we have:
 \begin{equation*}
n_{k+1}\wedge(n_0\vee n_1\vee\cdots\vee n_k)=(n_{k+1}\wedge n_0)\vee(n_{k+1}\wedge n_1)\vee\cdots\vee(n_{k+1}\wedge n_k),
\end{equation*}
where $n_i\vee n_j=\max\{n_i,n_j\}$ and $n_i\wedge n_j=\min\{n_i,n_j\}$.
\end{lemma}
\begin{proof}
First we take:
$$\beta=(n_{k+1}\wedge n_0)\vee(n_{k+1}\wedge n_1)\vee\cdots\vee(n_{k+1}\wedge n_k)\ \text{and}\ \alpha=n_{k+1}\wedge(n_0\vee n_1\vee\cdots\vee n_k).$$
Without loss of generality, we can assume that $n_0\geqslant n_1\geqslant\cdots\geqslant n_k$. There are three cases to be considered:

\vspace{0.5cm}
\noindent
$(a)$
$n_{k+1}\geqslant n_0$; for which we have
$$\alpha=n_0=\beta.$$
$(b)$
$n_j\geqslant n_{k+1}\geqslant n_{j+1}$ for some $0\leqslant j<k$; for which we have
$$\alpha=n_{k+1}=\beta.$$
$(c)$
$ n_k\geqslant n_{k+1}$; for which we also have
$$\alpha=n_{k+1}=\beta.$$
\qed\end{proof}

\begin{lemma}\rm\label{m}
For integers $n_0,n_1,\cdots,n_k$, let $n$ be the least common multiplier of $n_0,\cdots,n_k$ and $d_{i,j}$ be the greatest common divisor of $n_i$ and $n_j$ for $i\neq j$. Then the greatest common divisor of integers $n$ and $n_{k+1}$ is the least common multiplier of $d_{0,k+1},\cdots,d_{k,k+1}$.
\end{lemma}
\begin{proof}
Suppose that $\rho_0,\rho_1,\rho_2,\cdots$ is the sequence of all prime numbers $(2,3,5,\cdots)$.
If $n_j=\prod_i\rho_i^{m_i(j)}$ for $j=0,1,\cdots,k+1$, then
$$[n_0,n_1,n_2,\cdots,n_k]=\prod_{i}\rho_i^{m_i(0)\vee m_i(1)\vee\cdots\vee m_i(k)}$$
and
$$d_{j,k+1}=(n_j,n_{k+1})=\prod_{i}\rho_i^{m_i(j)\wedge m_i(k+1)}.$$
So, by Lemma~\ref{max-min}:
\begin{center}
\begin{tabular}{lcl}
$(n_{k+1},\,[n_0,n_1,n_2,\cdots,n_k])$
&\!\!\!\!$=$\!\!\!\!&$\prod_{i}\rho_i^{m_i(k+1)\wedge (m_i(0)\vee m_i(1)\vee\cdots\vee m_i(k))}$
\\
&\!\!\!\!$=$\!\!\!\!&$\prod_{i}\rho_i^{(m_i(k+1)\wedge m_i(0))\vee (m_i(k+1)\wedge m_i(1))\vee\cdots\vee (m_i(k+1)\wedge m_i(k))}$
\\
&\!\!\!\!$=$\!\!\!\!&$[(n_0,n_{k+1}),(n_1,n_{k+1}),\cdots,(n_{k},n_{k+1})]$
\\
&\!\!\!\!$=$\!\!\!\!&$[d_{0,k+1},d_{1,k+1},\cdots,d_{k,k+1}]$.
\end{tabular}
\end{center}
\qed\end{proof}

\begin{pro}[The Generalized Chinese Remainder]\label{crt}{\rm
For integers $t_0,t_1,\cdots,t_k$ and $n_0,n_1,\cdots,n_k\geqslant 2$, we have:
$$\exists x(\bigwedge\hspace{-2.3ex}\bigwedge_{i=0}^{k}x\equiv_{n_i} t_i)\iff\bigwedge\hspace{-4.3ex}\bigwedge_{0\leqslant i<j\leqslant k}t_i \equiv_{d_{i,j}} t_j$$
where $d_{i,j}$ is the greatest common divisor of $n_i$ and $n_j$ for $i\neq j;$ see ~\cite{frankel}.
}\end{pro}
\begin{proof}
The `only if' part is easy: For integers $t_0,t_1,\cdots,t_k$ and $n_0,n_1,\cdots,n_k\geqslant 2$,
suppose that there exists some $x$ such that $x\equiv_{n_i} t_i$ holds for $i=0,\cdots,k$. By $d_{i,j}\mid n_j$ and $d_{i,j}\mid n_i$ for $i\neq j$, we have:
$$x\equiv_{d_{i,j}} t_j\hspace{0.5cm}\text{and}\hspace{0.5cm}x\equiv_{d_{i,j}} t_i.$$
And so, $t_i\equiv_{d_{i,j}} t_j$.\\
We prove the `if' part by induction on $k$. For $k=0$ there is nothing to prove, and for $k=1$ we note that by Lemma~\ref{bz}, there are $a_0,a_1$ such that
\begin{equation}\label{chr-1}
a_0n_0 + a_1n_1 = d_{0,1}.
\end{equation}
Also, by the assumption there exists some $c$ such that
\begin{equation}\label{chr-2}
t_0-t_1=cd_{0,1}.
\end{equation}
Now, if we take $x$ to be $a_0(n_0/d_{0,1})t_1 + a_1(n_1/d_{0,1})t_0$, then by \eqref{chr-1} and \eqref{chr-2} we have
$$x=t_0-a_0n_0c\hspace{0.5cm}\text{and}\hspace{0.5cm}x=t_1+a_1n_1c.$$
And so  we have:
$$x\equiv_{n_0}t_0\hspace{0.5cm}\text{and}\hspace{0.5cm}x\equiv_{n_1}t_1.$$
For the induction step ($k+1$) we note that by the assumption, $t_i \equiv_{d_{i,j}} t_j$ holds for each $0\leqslant i<j\leqslant k+1$, and suppose that the following relations hold for some integer $x$ (the induction hypothesis):
\begin{equation}\label{chr-3}
\begin{cases}
x\equiv_{n_0} t_0 \\
x\equiv_{n_1} t_1 \\
\ \ \ \ \vdots\\
x\equiv_{n_k} t_k
\end{cases}
\end{equation}
Let $n$ be the least common multiplier of $n_0,\cdots,n_k$; then the greatest common divisor  $m$  of $n$ and $n_{k+1}$ is the least common multiplier of $d_{0,k+1},\cdots,d_{k,k+1}$ by Lemma~\ref{m}.\\
Now, by \eqref{chr-3} we have:
\begin{equation}\label{chr-4}
\begin{cases}
x\equiv_{d_{0,k+1}}t_0 \\
x\equiv_{d_{1,k+1}}t_1 \\
\ \ \ \ \vdots\\
x\equiv_{d_{k,k+1}}t_k
\end{cases}
\end{equation}
and by the assumption we have:
\begin{equation}\label{chr-5}
\begin{cases}
t_0\equiv_{d_{0,k+1}}t_{k+1} \\
t_1\equiv_{d_{1,k+1}}t_{k+1} \\
\ \ \ \ \vdots\\
t_k\equiv_{d_{k,k+1}}t_{k+1}
\end{cases}
\end{equation}
so by \eqref{chr-4} and \eqref{chr-5}
\begin{equation}\label{chr-6}
\begin{cases}
x\equiv_{d_{0,k+1}}t_{k+1} \\
x\equiv_{d_{1,k+1}}t_{k+1} \\
\ \ \ \ \vdots\\
x\equiv_{d_{k,k+1}}t_{k+1}
\end{cases}
\end{equation}
thus
$x\equiv_m t_{k+1}$ holds by \eqref{chr-6} and so, for some $c$ we have:
\begin{equation}\label{chr-7}
x-t_{k+1}=mc.
\end{equation}
By Lemma~\ref{bz}, there are $a,b$ such that
\begin{equation}\label{chr-8}
an+bn_{k+1}=m.
\end{equation}
Now, by \eqref{chr-7} and \eqref{chr-8} for $y=x-anc$, we have:
$$y=t_{k+1}+bn_{k+1}c\equiv_{n_{k+1}}t_{k+1}.$$
And also $y\equiv_{n_i}x\equiv_{n_i}t_i$ holds for each $0\leqslant i\leqslant k$.
\qed\end{proof}
%----------------------------------------------------------------------------------------
%	SECTION 4
%----------------------------------------------------------------------------------------
\section{Integer Numbers with Order and Addition}
\subsection{Quantifier Elimination of $\langle\mathbb{Z};<,+\rangle$}
Theorem \ref{thm-oaz} has been proved, in various formats, in e.g. the following references:
\cite[Chapter 24]{bbj}, \cite[Theorem 32E]{enderton}, \cite[Corollary 2.5.18]{hinman}, \cite[Secion III, Chapter 4]{kk}, \cite[Corollary 3.1.21]{marker}, \cite[Theorem 13.10]{monk} and  \cite[Section 4, Chapter III]{smorynski}.

 \noindent$\bullet$
 Here, we present a slightly different proof.
 \smallskip
 \begin{convention}\rm
 The Axioms of the Theory of Non-trivial Discretely Ordered Abelian Groups with the Division Algorithm are as follows:
 \begin{equation*}
\begin{tabular}{l}
($\texttt{O}_1$) \; $\forall x,y(x<y\rightarrow y\not<x)$  \\
($\texttt{O}_2$) \; $\forall x,y,z (x<y<z\rightarrow x<z)$  \\
($\texttt{O}_3$) \; $\forall x,y (x<y \vee x=y \vee y<x)$ \\
($\texttt{A}_1$) \; $\forall x,y,z\,(x+(y+z)=(x+y)+z)$ \\
($\texttt{A}_2$) \; $\forall x (x+\mathbf{\mathbf 0}=x)$ \\
($\texttt{A}_3$) \; $\forall x (x+ (-x)=\mathbf{\mathbf 0})$ \\
($\texttt{A}_4$) \; $\forall x,y (x+y=y+x)$ \\
($\texttt{A}_5$) \; $\forall x,y,z(x<y\rightarrow x+z<y+z)$ \\
($\texttt{O}_7^\circ$) \; $\forall x,y \big(x<y\leftrightarrow x+{\mathbf 1}\leqslant y\big)$ \\
($\texttt{A}_7^\circ$)  \; $\forall x\exists  y\,\big(\!\bigvee\hspace{-1.5ex}\bigvee_{i<n}x=n\centerdot y + \bar{i}\,\,\big)$   %\qquad \qquad \qquad
\hspace{1em} $n\in\mathbb{N}^+,\,\,\,\,\bar{i}=
\underbrace{{\mathbf 1}+\cdots+{\mathbf 1}}_{\text{i-times}}$
\end{tabular}
\end{equation*}
 \qecon\end{convention}

 \begin{pro}\rm The theory of the structure $\langle\mathbb{Z};<,+,-,{\bf 0},{\bf 1}\rangle$ does not admit quantifier elimination.
 \end{pro}
 \begin{proof}
  It suffices to show that the formula $\exists x (x+x=y)$ is not equivalent with any quantifier-free formula. All the terms including the free variable $y$ in the language $\langle+,-,{\bf 0},{\bf 1}\rangle$ are equal to $m.y$ for some $m\in\mathbb{Z}$, so all the atomic formulas are $m.y=k$, $m.y>k$ or $m.y<k$, for some $m,k\in\mathbb{Z}$. It is easy to see that all the definable sets of the above structure are finite or co-finite, whereas the set $\{y\in\mathbb{Z}\mid\exists x (x+x=y)\}$ is neither finite nor co-finite.
 \qed\end{proof}

\begin{theorem}\rm\label{thm-oaz}
The  infinite theory of non-trivial discretely ordered abelian groups with the division algorithm, that is $\texttt{O}_1$, $\texttt{O}_2$, $\texttt{O}_3$, $\texttt{A}_1$, $\texttt{A}_2$, $\texttt{A}_3$, $\texttt{A}_4$, $\texttt{A}_5$, $\texttt{O}_7^\circ$, $\texttt{A}_7^\circ$,
completely axiomatizes the order and additive theory of the   integer  numbers and, moreover, the \textup{(}theory of the\textup{)} structure $\langle\mathbb{Z};<,+,-,{\bf 0},{\bf 1},\{\equiv_n\}_{n\geqslant 2}\rangle$  admits quantifier elimination,   so has a decidable theory.
\end{theorem}
\begin{proof}
Indeed, the axiom  $\texttt{A}_7^\circ$ is   equivalent with
$$\forall x\bigvee\hspace{-2.2ex}\bigvee_{i<n} \big(x\equiv_n\bar{i}\wedge \bigwedge\hspace{-3.3ex}\bigwedge_{i\neq j<n}x\not\equiv_n\bar{j}\big),$$
\medskip
which is rather easy to verify,
and so the negation signs behind the congruences can be eliminated by
$$(a\not\equiv_n b) \leftrightarrow \bigvee\hspace{-3.3ex}\bigvee_{0<i<n} (a\equiv_n b+\bar{i}\,).$$

\medskip

Since every term containing the variable $x$ is equal to $n\centerdot x + t$, for some $x$-free term $t$ and $n\!\in\!\mathbb{Z}\!-\!\{0\}$, every atomic formula containing $x$ is equivalent with $n\centerdot x \Box t$ where $\Box\!\in\!\{=,<,>,\{\equiv_n\}_{n\geqslant 2}\}$ and $t$ is an $x$-free term. Whence, by Remark~\ref{mainlem}, it suffices to prove the equivalence of the formula

\begin{equation}\label{r-3}
\exists x (\bigwedge\hspace{-2.35ex}\bigwedge_{i<m} a_i\centerdot x\equiv_{n_i}t_i \;\wedge\;   \bigwedge\hspace{-2.35ex}\bigwedge_{j<p} u_j\!<\!b_j\centerdot x \;\wedge\;  \bigwedge\hspace{-2.3ex}\bigwedge_{k<q} c_k\centerdot x\!<\!v_k \;\wedge\;
   \bigwedge\hspace{-2.25ex}\bigwedge_{\ell<r} d_\ell\centerdot x=w_\ell)
\end{equation}

\medskip

with some quantifier-free formula, where $a_i$'s, $b_j$'s, $c_k$'s and $d_\ell$'s are natural numbers and $t_i$'s, $u_j$'s, $v_k$'s and $w_\ell$'s are $x$-free terms.

By the equivalences
\begin{itemize}\itemindent=5ex
\item[(i)] $[a<b] \leftrightarrow [n\centerdot a< n\centerdot b]$,
\item[(ii)] $[a=b] \leftrightarrow [n\centerdot a=n\centerdot b]$,
\item[(iii)]  $[a\equiv_m b] \leftrightarrow [n\centerdot a\equiv_{nm}n\centerdot b]$,
\end{itemize}
which are provable from the axioms, we can assume that $a_i$'s, $b_j$'s, $c_k$'s and $d_\ell$'s  in
the formula~\eqref{r-3} are equal to each other, say to $\alpha$. Now, \eqref{r-3} is equivalent with
\begin{equation}\label{r-4}
\exists y (y\equiv_\alpha{\bf 0} \;\wedge\; \bigwedge\hspace{-2.35ex}\bigwedge_{i<m} y\equiv_{n_i}t_i' \;\wedge\;   \bigwedge\hspace{-2.35ex}\bigwedge_{j<p} u_j'\!<\!y \;\wedge\;  \bigwedge\hspace{-2.3ex}\bigwedge_{k<q} y\!<\!v_k' \;\wedge\;
   \bigwedge\hspace{-2.25ex}\bigwedge_{\ell<r} y=w_\ell'),
\end{equation}
for $y=\alpha\centerdot x$ and some (possibly new) terms $t_i'$'s, $u_j'$'s, $v_k'$'s and $w_\ell'$'s.

\noindent
If $r\neq 0$, then \eqref{r-4} is readily equivalent with the  quantifier-free formula which results from  substituting $w_0'$ with $y$.
So, it suffices to eliminate the quantifier of
\begin{equation}\label{r-5}
\exists x (\bigwedge\hspace{-2.35ex}\bigwedge_{i<m} x\equiv_{n_i}t_i \;\wedge\;   \bigwedge\hspace{-2.35ex}\bigwedge_{j<p} u_j\!<\!x \;\wedge\;  \bigwedge\hspace{-2.3ex}\bigwedge_{k<q} x\!<\!v_k).
\end{equation}
By the equivalence of the formula $\exists x(\theta(x)\wedge u_0\!<\!x \wedge u_1\!<\!x)$ with the   formula
 $$\big[\exists x (\theta(x)\!\wedge\!u_0\!<\!x)\!\wedge\!u_1\!\leqslant\!u_0\big]
\!\vee\!\big[\exists x (\theta(x)\!\wedge\!u_1\!<\!x)\!\wedge\!u_0\!\leqslant\!u_1\big],$$
we can assume that $p\leqslant 1$ (and   $q\leqslant 1$ by a dual argument).
 Also, the following formula with two $x$-congruences
 $$\exists x(\theta(x)\wedge x\equiv_{n_0}t_0  \wedge x\equiv_{n_1}t_1)$$
 is equivalent with the following formula with just one $x$-congruence
$$\exists x(\theta(x)\wedge x\equiv_{n}t)  \wedge t_0\equiv_{d}t_1,$$
where $d$ is the greatest common divisor of $n_0$ and $n_1$, $n$ is their least common multiplier, and  $t=a_0(n_0/d)t_1 + a_1(n_1/d)t_0$ where $a_0,a_1$ satisfy $a_0n_0 + a_1n_1 = d$ (see the proof of Proposition~\ref{crt}). So, we can assume that $m\leqslant 1$ as well.

\noindent
Now, if $m=0$ then the formula~\eqref{r-5} is equivalent with a quantifier-free formula by Theorem~\ref{thm-oz} (with $\mathfrak{s}(x)=x+{\bf 1}$ just like the way formula~\eqref{f-4} was equivalent with some quantifier-free formula).

\noindent
So, suppose   $m=1$. In this case, if any of $p$ or $q$ is equal to $0$ then \eqref{r-5} is equivalent with $\top$ (since any congruence can have infinitely large or infinitely small solutions).

\noindent
Finally, if we have $p=q=1=m$, then the formula
$
\exists x (x\equiv_{n}t \;\wedge\;   u\!<\!x \;\wedge\;  x\!<\!v)
$
is equivalent with the formula $\exists y (r<n\centerdot y\leqslant s)$ for $x=t+n\centerdot y$, $r=u-t$ and $s=v-t-{\bf 1}$. Now, the formula
$\exists y (r<n\centerdot y\leqslant s)$  is   equivalent with   the quantifier-free formula $\bigvee\hspace{-2.25ex}\bigvee_{i<n}(s\equiv_n \bar{i} \;\wedge\;  r+\bar{i}<s)$, since there are some $q$ and some $i<n$ such that $s=qn + i$. The existence of some $y$ such that $r<n\centerdot y\leqslant s$   is then equivalent with   $r<nq$ ($=s-i$).
\qed\end{proof}

\subsection{Non-finite Axiomatizability of $\langle\mathbb{Z};<,+\rangle$}
\begin{pro}\rm\label{rem-infz}
The theory of $\langle\mathbb{Z};<,+\rangle$ cannot be axiomatized finitely.
\end{pro}
\begin{proof}
We show that $\texttt{O}_1$, $\texttt{O}_2$, $\texttt{O}_3$, $\texttt{A}_1$, $\texttt{A}_2$, $\texttt{A}_3$, $\texttt{A}_4$, $\texttt{A}_5$, $\texttt{O}_7^\circ$ and any finite number of the instances of $\texttt{A}_7^\circ$ cannot prove all the instances of $\texttt{A}_7^\circ$. To see this take $\mathfrak{p}$ to be a sufficiently large prime number and put $N=(\mathfrak{p}-1)!$. Let us recall that the (rational) set   $\mathbb{Q}/N=\{m/N^k\mid m\in\mathbb{Z},k\in\mathbb{N}\}$ (Theorem~\ref{rem-infq}) is closed under the addition operation  and   $x\mapsto x/n$ for any $1<n<\mathfrak{p}$. Define the set  $\mathcal{A}=(\mathbb{Q}/N)\times\mathbb{Z}$ and put  the structure $\mathfrak{A}=\langle\mathcal{A};<_\mathfrak{A},+_\mathfrak{A},
-_\mathfrak{A},{\bf 0}_\mathfrak{A},{\bf 1}_\mathfrak{A}\rangle$ on it by  the following:
\begin{itemize}\itemindent=2em
\item[$(<_\mathfrak{A})$:] $(a,\ell)<_\mathfrak{A}(b,m)\iff (a<b)\vee (a=b\wedge \ell<m)$;
\item[$(+_\mathfrak{A})$:] $(a,\ell)+_\mathfrak{A}(b,m)=(a+b,\ell+m)$;
\item[$(-_\mathfrak{A})$:] $-_\mathfrak{A}(a,\ell)=(-a,-\ell)$;
\item[$({\bf 0}_\mathfrak{A})$:] ${\bf 0}_\mathfrak{A}=({\mathbf0},{\mathbf0})$;
\item[$({\bf 1}_\mathfrak{A})$:] ${\bf 1}_\mathfrak{A}=({\mathbf0},{\mathbf1})$.
\end{itemize}
It is straightforward to see that $\mathfrak{A}$ satisfies the axioms $\texttt{O}_1$, $\texttt{O}_2$, $\texttt{O}_3$, $\texttt{A}_1$, $\texttt{A}_2$, $\texttt{A}_3$, $\texttt{A}_4$, $\texttt{A}_5$ and $\texttt{O}_7^\circ$; but does not satisfy $\texttt{A}_7^\circ$ for $n=\mathfrak{p}$ since  the equality   $({\mathbf1},{\mathbf0})=\mathfrak{p}\centerdot(a,\ell)+\bar{i}$ for any  $a\in\mathbb{Q}/N,\ell\in\mathbb{Z},i\in\mathbb{N}$ (with $i<\mathfrak{p}$) implies that $a=1/\mathfrak{p}$ but $1/\mathfrak{p}\not\in\mathbb{Q}/N$. However, $\mathfrak{A}$ satisfies the finite number of the instances of $\texttt{A}_7^\circ$ (for any $1<n<\mathfrak{p}$): for any element  $(a,\ell)\in\mathcal{A}$ we have $a=m/N^k$ for some $m\in\mathbb{Z}$, $k\in\mathbb{N}$, and $\ell=nq+r$ for  some $q,r$ with $0\leqslant r<n$; now, $(a,\ell)=n\centerdot \big(m'/N^{k+1},q\big)+_\mathfrak{A}(0,r)$ (where $m'=m\cdot (N/n)\in\mathbb{Z}$) and so $(a,\ell)=n\centerdot \big(m'/N^{k+1},q\big)+_\mathfrak{A}\bar{r}$ (where $\bar{r}={\bf 1}_\mathfrak{A}+_\mathfrak{A}\cdots+_\mathfrak{A}{\bf 1}_\mathfrak{A}$ for $r$ times).
\qed\end{proof}

%----------------------------------------------------------------------------------------
%	SECTION 5
%----------------------------------------------------------------------------------------

\section{Natural Numbers with Order and Addition}

\subsection{Axiomatization of $\langle\mathbb{N};<,+\rangle$}
\begin{theorem}\rm
The following axioms completely axiomatize the theory of the structure of $\langle\mathbb{N};<,+,{\mathbf 0},{\mathbf 1}\rangle$:
\begin{equation*}
\begin{tabular}{l}
($\texttt{O}_1$)\; $\forall x,y(x<y\rightarrow y\not<x)$  \\
($\texttt{O}_2$) \; $\forall x,y,z (x<y<z\rightarrow x<z)$  \\
($\texttt{O}_3$) \; $\forall x,y (x<y \vee x=y \vee y<x)$ \\
($\texttt{O}_7$) \; $\forall x,y (x\!<\!y \;\leftrightarrow\; x+{\mathbf 1}\!<\!y \vee x+{\mathbf 1}\!=\!y)$ \\
($\texttt{O}_8^\circ$) \; $\forall x\exists y (x\neq {\mathbf 0} \rightarrow y+{\mathbf 1}=x)$ \\
($\texttt{O}_9$) \; $\forall x (x\not<{\mathbf 0}),$ \\
($\texttt{A}_1$) \; $\forall x,y,z\,(x+(y+z)=(x+y)+z)$ \\
($\texttt{A}_2$) \; $\forall x (x+\mathbf{\texttt 0}=x)$ \\
($\texttt{A}_4$) \; $\forall x,y (x+y=y+x)$ \\
($\texttt{A}_5$) \; $\forall x,y,z(x<y\rightarrow x+z<y+z)$ \\
%$\forall x,y \big(x<y\leftrightarrow x+{\bf 1}\leqslant y\big)$ \; ($\texttt{O}_7^\circ$) \\
($\texttt{A}_7^\circ$) \;  $\forall x\exists  y\,\big(\bigvee\hspace{-1.5ex}\bigvee_{i<n}x=n\centerdot y + \bar{i}\,\big)$  \qquad \qquad \qquad  $\hspace{-2cm}n\in\mathbb{N}^+,\,\,\,\,\bar{i}=\underbrace{{\mathbf 1}+\cdots+{\mathbf 1}}_{\text{i-times}}$
\end{tabular}
\end{equation*}
and, moreover, the structure $\langle\mathbb{N};<,+,{\mathbf 0},\{\equiv_n\}_{n\geqslant 2}\rangle$ admits quantifier elimination, and so its theory is decidable.
\end{theorem}
\begin{proof}
The quantifier elimination of this structure is shown in \cite[Theorem~32E]{enderton}.
\qed\end{proof}

\subsection{Decidability of $\langle\mathbb{N};<,+\rangle$}
Here, we use the super-structure $\langle\mathbb{Z};<,+\rangle$ to show the decidability of the theory of natural numbers with order and addition.
\begin{remark}\rm\label{nz}
The set of natural numbers is definable in structure $\langle\mathbb{Z};<,+\rangle$ by
$$``x\in\mathbb{N}"\iff\exists y (y\!+\!y\!=\!y\wedge y\leqslant x).$$
\qecon\end{remark}
\begin{theorem}\rm\label{rem-noa}
The theory of the structure $\langle\mathbb{N};<,+\rangle$ is decidable.
\end{theorem}
\begin{proof}
We show that the decidability of the structure $\langle\mathbb{Z};<,+\rangle$ implies the decidability of the structure $\langle\mathbb{N};<,+\rangle$. Relativization $\psi^{\mathbb{N}}$ of a $\{<,+\}$-formula $\psi$ resulted  from substituting any subformula of the form $\forall x\theta(x)$ by $\forall x [``x\in\mathbb{N}"\!\rightarrow\!\theta(x)]$ and $\exists x\theta(x)$ by $\exists x [``x\in\mathbb{N}"\!\wedge\!\theta(x)]$ by Remark~\ref{nz}  has the following property: $$\langle\mathbb{N};<,+\rangle\models\psi \iff \langle\mathbb{Z};<,+\rangle\models\psi^{\mathbb{N}}.$$
So, the theory of the structure $\langle\mathbb{N};<,+\rangle$ is decidable
\qed\end{proof}

% Chapter Template

\chapter{Multiplicative Ordered Structures} % Main chapter title
\label{Chapter4} % Change X to a consecutive number; for referencing this chapter elsewhere, use \ref{ChapterX}
In this chapter we consider the theories of the number sets $\mathbb{N},\mathbb{Z},\mathbb{R}$ and $\mathbb{Q}$ over the language $\{<,\times\}$.
%\subsection{\bf Natural Numbers with Order and Multiplication}\label{subsec-n}
%----------------------------------------------------------------------------------------
%	SECTION 1
%----------------------------------------------------------------------------------------

\section{Natural numbers with order and multiplication}\label{nom}
\subsection{Non-Axiomatizability of $\langle\mathbb{N};<,\times\rangle$}
\begin{pro}\rm\label{subsec-n}
The theory of the structure $\langle\mathbb{N};<,\times\rangle$ is undecidable.
\end{pro}
\begin{proof}
First we notice that the addition operation is definable in $\langle\mathbb{N};<,\times\rangle$, since
    \begin{itemize}
    \item[$(1)$] successor  $\mathfrak{s}$ is definable from $<$:
         $$y\!=\!\mathfrak{s}(x) \iff x\!<\!y \wedge \neg\exists z (x\!<\!z\!<\!y);$$
    \item[$(2)$] and  addition  is definable from the successor and multiplication: \\
    \hfill$z\!=\!x\!+\!y  \iff$
        \newline{$\big[\neg\exists u (\mathfrak{s}(u)\!=\!z)\wedge x\!=\!y\!=\!z\big] \vee\big[\exists u (\mathfrak{s}(u)\!=\!z) \wedge \mathfrak{s}(z\cdot x)\cdot \mathfrak{s}(z\cdot y)=
\mathfrak{s}(z\cdot z\cdot \mathfrak{s}(x\cdot y))\big].$}
    \end{itemize}
(The above identity\label{r} was first introduced by Robinson~\cite{robinson}; also see e.g.   \cite[Chapter 24]{bbj} or \cite[Exercise 2 on page 281]{enderton}.)\\
Now by (1) and (2), the structure $\langle\mathbb{N};<,\times\rangle$ can interpret  the structure  $\langle\mathbb{N};+,\times\rangle$ whose theory is  undecidable by G\"odel's Incompleteness theorem. Thus, the theory of the structure $\langle\mathbb{N};<,\times\rangle$ is undecidable (see \cite[Theorem 17.4]{bbj}, \cite[Corollary 35A]{enderton}, \cite[Theorem 4.1.7]{hinman}, \cite[Chapter 15]{monk} or \cite[Corollary 6.4 in Chapter III]{smorynski} for a proof of the undecidability of the structure $\langle\mathbb{N};<,\times\rangle$ and some more details).
\qed\end{proof}
\begin{cor} \rm
The structure $\langle\mathbb{N};<,\times\rangle$ can not be axiomatized by any computably enumerable set of sentences.
\qed\end{cor}

%----------------------------------------------------------------------------------------
%	SECTION 2
%----------------------------------------------------------------------------------------

\section{Integer numbers with order and multiplication}
\subsection{Non-Axiomatizability of $\langle\mathbb{Z};<,\times\rangle$}

The undecidability of the theory of the structure $\langle\mathbb{N};+,\times\rangle$ also implies the undecidability of the theories of the structures
$\langle\mathbb{Z};+,\times\rangle$ and $\langle\mathbb{Z};<,\times\rangle$.
\begin{pro}\rm\label{zam}
The theory of the structure $\langle\mathbb{Z};+,\times\rangle$ is undecidable.
\end{pro}
\begin{proof}
By Lagrange's Four Square Theorem (see e.g. \cite[Theorem 16.6]{monk}) $\mathbb{N}$ is definable in $\langle\mathbb{Z};+,\times\rangle$:
$$u\in\mathbb{N} \iff \exists x,y,z,t (u=x\cdot x+y\cdot y+z\cdot z+t\cdot t).$$
Whence, $\langle\mathbb{N};+,\times\rangle$ is definable in $\langle\mathbb{Z};+,\times\rangle$,
  and so $\langle\mathbb{Z};+,\times\rangle$ has an undecidable theory by G\"odel's Incompleteness theorem (see e.g.  \cite[Theorem 16.7]{monk} or \cite[Corollary 8.29 in Chapter III]{smorynski}).
\qed\end{proof}
\begin{pro}\rm\label{subsec-z}
The theory of the structure $\langle\mathbb{Z};<,\times\rangle$ is undecidable.
\end{pro}
\begin{proof}
First we notice that the following numbers and  operations   are definable in the structure $\langle\mathbb{Z};<,\times\rangle$:
   \begin{itemize}
   \item[--] The number zero:
   $$u={\mathbf 0} \iff \forall x (x\cdot u = u).$$
   \item[--] The number one:
   $$u={\mathbf 1} \iff \forall x (x\cdot u =x).$$
   \item[--] The number $-{\mathbf 1}$:
   $$u=-{\mathbf 1} \iff u\cdot u={\mathbf 1} \wedge u\neq {\mathbf 1}.$$
   \item[--]  The additive inverse:
   $$y=-x \iff y=(-{\mathbf 1})\cdot x.$$
   \item[--] The successor:
   $$y=\mathfrak{s}(x) \iff x<y \wedge \neg\exists z (x<z<y).$$
   \item[--] The addition:
   $$z=x+y \iff [z={\mathbf 0}\wedge y=-x] \vee [z\neq {\mathbf 0}\wedge \mathfrak{s}(z\cdot x)\cdot \mathfrak{s}(z\cdot y)=
\mathfrak{s}(z\cdot z\cdot \mathfrak{s}(x\cdot y))].$$
   \end{itemize}
There is another beautiful
  definition for $+$ in terms of $\mathfrak{s}$ and $\times$ in $\mathbb{Z}$ in \cite[p.~187]{hinman}:\\

    \hfill     $z=x+y \iff$\\
    \hfill$[z\cdot\mathfrak{s}(z)=z\wedge\mathfrak{s}(x\cdot y)=\mathfrak{s}(x)\cdot\mathfrak{s}(y)]
    \vee[z\cdot\mathfrak{s}(z)\neq z\wedge \mathfrak{s}(z\cdot x)\cdot \mathfrak{s}(z\cdot y)=
\mathfrak{s}(z\cdot z\cdot \mathfrak{s}(x\cdot y))]$.

\noindent And so, the structure $\langle\mathbb{Z};+,\times\rangle$ somehow includes the structure $\langle\mathbb{Z};<,\times\rangle$.
By Proposition~\ref{zam}, the theory of the structure $\langle\mathbb{Z};+,\times\rangle$ is undecidable. Thus the theory of the structure $\langle\mathbb{Z};<,\times\rangle$ is undecidable too.
\qed\end{proof}
\begin{cor}\rm
The structure $\langle\mathbb{Z};<,\times\rangle$ can not be axiomatized by any computably enumerable set of sentences.
\qed\end{cor}

%----------------------------------------------------------------------------------------
%	SECTION 3
%----------------------------------------------------------------------------------------

\section{Real numbers with order and multiplication}\label{subsec-r}
The structure $\langle\mathbb{R};<,\times\rangle$ is decidable since by a theorem of Tarski  the (theory of the)
structure $\langle\mathbb{R};<,+,\times\rangle$ can be completely axiomatized by the theory of {\em real closed ordered fields}, and so has a decidable theory; see e.g. \cite[Theorem 7, Chapter 4]{kk}, \cite[Theorem 3.3.15]{marker} or \cite[Theorem 21.36]{monk}.
\begin{cor}\rm
For the reason that the structure $\langle\mathbb{R};<,\times\rangle$ is included in the structure $\langle\mathbb{R};<,+,\times\rangle$, the theory of the structure $\langle\mathbb{R};<,\times\rangle$ is also decidable.
\qed\end{cor}
$\bullet$
Here, we prove the decidability of this theory directly (without using Tarski's theorem) and provide an explicit axiomatization for it.

\subsection{Axiomatization and Quantifier Elimination of $\langle\mathbb{R};<,\times\rangle$}
First we study the structure $\langle\mathbb{R}^+;<,\times\rangle$.
\begin{pro}\rm\label{thm-or+}
The following infinite theory \textup{(}of the non-trivial ordered divisible abelian groups\textup{)}
completely axiomatizes the order and multiplicative theory of the positive real numbers:

\begin{tabular}{l}
($\texttt{O}_1$) \; $\forall x,y(x<y\rightarrow y\not<x)$  \\ ($\texttt{O}_2$) \; $\forall x,y,z (x<y<z\rightarrow x<z)$ \\
($\texttt{O}_3$) \; $\forall x,y (x<y \vee x=y \vee y<x)$ \\
($\texttt{M}_1$) \; $\forall x,y,z\,(x\cdot(y\cdot z)=(x\cdot y)\cdot z)$ \\
($\texttt{M}_2$) \; $\forall x (x\cdot\mathbf{1}=x)$ \\
($\texttt{M}_3$) \; $\forall x (x\cdot  x^{-1}=\mathbf{1})$ \\
($\texttt{M}_4$) \; $\forall x,y (x\cdot y=y\cdot x)$ \\
($\texttt{M}_5$) \; $\forall x,y,z(x<y\rightarrow x\cdot z<y\cdot z)$ \\
($\texttt{M}_6$) \; $\exists y (y\neq {\bf 1})$ \\
($\texttt{M}_7$) \; $\forall x\exists y(x=y^n)$  \qquad \qquad \qquad \qquad  $n\geqslant 2$
\end{tabular}

\noindent
The structure $\langle\mathbb{R}^+;<,\times,\square^{-1},{\bf 1}\rangle$  admits quantifier elimination, and so its theory is decidable.
\end{pro}
\begin{proof}
The structure $\langle\mathbb{R}^+;<,\times\rangle$ (of the positive real numbers) is (algebraically) isomorphic to the structure $\langle\mathbb{R};<,+\rangle$ by the mapping $x\mapsto\log(x)$. So, Theorem~\ref{thm-oa} implies the decidability of the structure $\langle\mathbb{R}^+;<,\times\rangle$.
\qed\end{proof}
\begin{pro}\rm\label{thm-or}
The following infinite theory completely axiomatizes the order and multiplicative theory of the   real  numbers:

\begin{tabular}{l}
($\texttt{O}_1$) \; $\forall x,y(x<y\rightarrow y\not<x)$  \\ ($\texttt{O}_2$) \; $\forall x,y,z (x<y<z\rightarrow x<z)$ \\
($\texttt{O}_3$) \; $\forall x,y (x<y \vee x=y \vee y<x)$ \\
($\texttt{M}_1$) \; $\forall x,y,z\,(x\cdot(y\cdot z)=(x\cdot y)\cdot z)$ \\
($\texttt{M}_2^\circ$) \; $\forall x (x\cdot\mathbf{1}=x  \;\,  \wedge \;\, x\cdot{\bf 0}={\bf 0}={\bf 0}^{-1})$ \\
($\texttt{M}_3^\circ$) \; $\forall x (x\neq {\bf 0}\rightarrow x\cdot  x^{-1}=\mathbf{1})$ \\
($\texttt{M}_4$) \; $\forall x,y (x\cdot y=y\cdot x)$ \\
($\texttt{M}_5^\circ$) \; $\forall x,y,z(x<y\wedge {\bf 0}<z \rightarrow x\cdot z<y\cdot z)$ \\
($\texttt{M}_5^\bullet$) \; $\forall x,y,z(x<y\wedge z<{\bf 0} \rightarrow y\cdot z<x\cdot z)$  \\
($\texttt{M}_6^\circ$) \; $\exists y ({\bf -1}<{\bf 0}<{\bf 1}<y)$
\\
($\texttt{M}_{7}^\circ$) \;  $\forall x\exists y (x=y^{2n+1})\qquad\qquad n\in\mathbb{N}$
\\
($\texttt{M}_{8}$) \;  $\forall
x (x^{2n}={\bf 1}\longleftrightarrow x={\bf 1}\vee x={\bf -1})\qquad\qquad n\in\mathbb{N}$ \\
($\texttt{M}_{9}$) \; $\forall x\,({\bf 0}<x\longleftrightarrow \exists y [y\neq{\bf 0}\wedge x=y^2])$
\end{tabular}

\noindent
and, moreover, the structure $\langle\mathbb{R};<,\times,\square^{-1},{\bf -1},{\bf 0},{\bf 1}\rangle$  admits quantifier elimination, and so its theory is decidable.
\end{pro}
\begin{proof}
We have $(x<{\bf 0}) \leftrightarrow ({\bf 0}<-x)$ by $\texttt{M}_{5}^\bullet$, $\texttt{M}_{2}^\circ$, $\texttt{M}_{6}^\circ$ and $\texttt{M}_{8}$, where $-x=({\bf -1})\cdot x$.
%So, by $\texttt{O}_{3}$, for any variable $z$ we have   three cases:    ${\bf 0}<z$, $z={\bf 0}$ or ${\bf 0}<-z$.
Whence, for any   formula $\eta$ we have
$$\exists x \eta(x)\equiv \exists x\!>\!{\bf 0}\eta(x)\vee \eta({\bf 0}) \vee \exists y\!>\!{\bf 0}\eta(-y).$$
Also, if $z$ is another variable in $\eta$, then $\eta(x,z)$ is equivalent with
$$[{\bf 0}<z \wedge \eta(x,z)]\vee \eta(x,{\bf 0}) \vee [{\bf 0}<-z\wedge\eta(x,z)].$$
For the last disjunct, if we let $z'=-z$, then ${\bf 0}<-z\wedge\eta(x,z)$ will be ${\bf 0}<z'\wedge\eta(x,-z')$.
Thus, by introducing the constants ${\bf 0}$ and ${\bf -1}$ (and renaming the variables if necessary) we can assume that all the variables of a quantifier-free formula  are positive.
 Now, the process of eliminating the quantifier of the formula $\exists x\eta(x)$, where $\eta$ is the conjunction of some atomic formulas (cf. Remark~\ref{mainlem}), goes as follows: \\
 We first eliminate the constants ${\bf 0}$ and ${\bf -1}$ and then reduce the desired conclusion to Proposition~\ref{thm-or+}.
 For the first part, we simplify terms so that each term is either positive (all the variables are positive) or equals to ${\bf 0}$ or is the negation of a positive term (is $-t$ for some positive term $t$). Then by replacing ${\bf 0}={\bf 0}$ with $\top$ and ${\bf 0}<{\bf 0}$ with $\bot$, we can assume that ${\bf 0}$ appears at most once in any atomic formula; also ${\bf -1}$ appears at most once since $-t=-s$ is equivalent with $t=s$ and $-t<-s$ with $s<t$.
 Now, we can eliminate the constant ${\bf -1}$ by replacing the atomic formulas $-t=s$, $t=-s$ and $t<-s$ by $\bot$ and $-t<s$ by $\top$ for positive or zero terms $t,s$ (note that   ${\bf -0}={\bf 0}$ by $\texttt{M}_{2}^\circ$).  Also the constant ${\bf 0}$ can be eliminated by replacing ${\bf 0}<t$ with $\top$ and  $t<{\bf 0}$ and $t={\bf 0}$ (also ${\bf 0}=t$) with $\bot$ for positive terms $t$.
 Thus, we get a formula whose all variables are positive, and so we are in the realm of $\mathbb{R}^+$. Finally, for the second part we have the equivalence of thus resulted formula with a quantifier-free formula by Proposition~\ref{thm-or+} provided that the relativized form of the axioms $\texttt{O}_{1}$, $\texttt{O}_{2}$, $\texttt{O}_{3}$, $\texttt{M}_{1}$, $\texttt{M}_{2}$, $\texttt{M}_{3}$, $\texttt{M}_{4}$, $\texttt{M}_{5}$, $\texttt{M}_{6}$ and $\texttt{M}_{7}$ to $\mathbb{R}^+$ can be proved from the axioms $\texttt{O}_{1}$, $\texttt{O}_{2}$, $\texttt{O}_{3}$, $\texttt{M}_{1}$, $\texttt{M}_{2}^\circ$, $\texttt{M}_{3}^\circ$, $\texttt{M}_{4}$, $\texttt{M}_{5}^\circ$, $\texttt{M}_{5}^\bullet$, $\texttt{M}_{6}^\circ$, $\texttt{M}_{7}^\circ$, $\texttt{M}_{8}$,  and $\texttt{M}_{9}$. We need to consider   $\texttt{M}_{6}$ and $\texttt{M}_{7}$ only, when relativized to $\mathbb{R}^+$, i.e., $\exists y({\bf 0}<y \wedge y\neq {\bf 1})$ and $\forall x \exists y [{\bf 0}<x \rightarrow {\bf 0}<y \wedge x=y^n]$. The relativization of $\texttt{M}_{6}$ immediately follows from  $\texttt{M}_{6}^\circ$. For the relativization of $\texttt{M}_{7}$ take any $a>{\bf 0}$, and any $n\in\mathbb{N}$. Write $n=2^k(2m+1)$; by $\texttt{M}_{7}^\circ$ there exists some $c$ such that $c^{2m+1}=a$, and by $\texttt{M}_{5}^\circ$ and $\texttt{M}_{5}^\bullet$ we should have $c>{\bf 0}$. Now, by using $\texttt{M}_{9}$  for $k$ times there must exist some $b$ such that $b^{2^k}=c$ and we can assume that $b>{\bf 0}$ (since otherwise we can take $-b$ instead of $b$). Now, we have $b^{2^k(2m+1)}=c^{2m+1}=a$ and so $a=b^n$.
 \qed\end{proof}

\subsection{Non-finite Axiomatizability of $\langle\mathbb{R};<,\times\rangle$}

\begin{pro}\rm
The structure $\langle\mathbb{R}^+;<,\times\rangle$ is not finitely axiomatizable.
\end{pro}
\begin{proof}
For the infinite axiomatizability it suffices to note that for a sufficiently large $N$, the set $\{2^{{m}\cdot(N!)^{-k}}\mid m\in\mathbb{Z},k\in\mathbb{N}\}$ of positive real numbers is a multiplicative subgroup and so satisfies all the axioms ($\texttt{O}_1$, $\texttt{O}_2$, $\texttt{O}_3$, $\texttt{M}_1$,
 $\texttt{M}_2$, $\texttt{M}_3$, $\texttt{M}_4$, $\texttt{M}_5$,
 $\texttt{M}_6$) and  finitely many instances of the axiom $\texttt{M}_7$ (for $n\leqslant N$) but not all the instances of $\texttt{M}_7$ (for example when $n=p$ is a prime larger than $N!$).
\qed\end{proof}
\begin{theorem}\rm
The structure $\langle\mathbb{R};<,\times\rangle$ is not finitely axiomatizable.
\end{theorem}
\begin{proof}
The set $\{0\}\cup\{-2^{{m}\cdot(N!)^{-k}},2^{{m}\cdot(N!)^{-k}}\mid m\in\mathbb{Z},k\in\mathbb{N}\}$ of real numbers, for some $N>2$, satisfies all the axioms of Theorem~\ref{thm-or} except $\texttt{M}_{7}^\circ$; however it satisfies a finite number of its instances (when $2n+1\leqslant N$) but not all the instances of $\texttt{M}_{7}^\circ$ (e.g. when $2n+1$ is a prime greater than $N!$).
%(cf. the proof of Proposition~\ref{thm-or+} and Remark~\ref{rem-infq}).
\qed\end{proof}

%----------------------------------------------------------------------------------------
%	SECTION 4
%----------------------------------------------------------------------------------------

\section{Rational numbers with order and multiplication}
The technique of the proof of Theorem~\ref{thm-or} enables us to consider first the multiplicative and order structure of the positive rational numbers, that is  $\langle\mathbb{Q}^+;<,\times\rangle$.

\subsection{Quantifier Elimination of $\langle\mathbb{Q};<,\times\rangle$}

\begin{pro}\rm
The theory of the structure $\langle\mathbb{Q}^+;<,\times\rangle$ does not admit quantifier elimination.
\end{pro}
\begin{proof}
We show that the formula $\exists x(y=x^n)$ (for $n>1$) is not equivalent with any quantifier-free formula. All the atomic formulas of the free variable $y$,  are $y^n<y^m$ or $y^n=y^m$ which do not depend on $y$ and are equivalent with $\top$ or $\bot$. So the formula $\exists x(y=x^n)$ (which depends on $y$ and $n$ and can be $\top$ or $\bot$) is not equivalent with any of them.
\qed\end{proof}
\begin{definition}[$\Re$]\label{def-re}\rm
Let  $\Re_n(y)$\label{rn} be the formula $\exists x(y=x^n)$, stating that ``$y$ is the $n$th power of a number'' (for $n>1$).
\qedef\end{definition}
\begin{remark}\rm
For any $r\in\mathbb{Q}$ and any natural $n>1$ the formula $\Re_n(r)$ holds if and only if every exponent of the unique factorization (of the numerators and denominators of the reduced form) of $r$ is divisible by $n$. Thus $\Re_n(r)$ is an algorithmically decidable relation of $r$ (and $n$).
\qecon\end{remark}
\begin{definition}[${\sf TQ}$]\label{def-tq}\rm
Let ${\sf TQ}$\label{tq} be the theory axiomatized by the axioms
\begin{equation*}
\begin{tabular}{l}
($\texttt{O}_1$) \; $\forall x,y(x<y\rightarrow y\not<x)$  \\
($\texttt{O}_2$) \; $\forall x,y,z (x<y<z\rightarrow x<z)$ \\
($\texttt{O}_3$) \; $\forall x,y (x<y \vee x=y \vee y<x)$ \\
($\texttt{M}_1$) \; $\forall x,y,z\,(x\cdot(y\cdot z)=(x\cdot y)\cdot z)$ \\
($\texttt{M}_2$) \; $\forall x (x\cdot\mathbf{\mathbf 1}=x)$ \\
($\texttt{M}_3$) \; $\forall x (x\cdot  x^{-1}=\mathbf{\mathbf 1})$ \\
($\texttt{M}_4$) \; $\forall x,y (x\cdot y=y\cdot x)$ \\
($\texttt{M}_5$) \; $\forall x,y,z(x<y\rightarrow x\cdot z<y\cdot z)$ \\
($\texttt{M}_6$) \; $\exists y (y\neq {\mathbf 1})$ \\
($\texttt{M}_{10}$) \;  $\forall x,z\exists y (x<z\rightarrow x<y^n<z)\qquad\qquad n\in\mathbb{N}$, and  \\
 ($\texttt{M}_{11}$) \;  $\forall \{x_j\}_{j<q} \exists y \forall z  \bigwedge\hspace{-1.5ex}\bigwedge_{m_j\nmid n (j<q)} (y^n\cdot x_j \neq z^{m_j})$ \quad  for each $n\geqslant 1$ (and $m_j>1$)
 \end{tabular}
\end{equation*}
\qedef\end{definition}
Some explanations on the new axioms  $\texttt{M}_{10}$ and $\texttt{M}_{11}$ are in order:

\noindent
The axiom $\texttt{M}_{10}$, interpreted in $\mathbb{Q}^+$,  states that   $\mathbb{Q}^+$ is dense not only in itself but also in the radicals of its elements (or more generally in $\mathbb{R}^+$: for any $x,z\in\mathbb{Q}^+$ there exists some $y\in\mathbb{Q}^+$ that satisfies $\sqrt[n]{x}<y<\sqrt[n]{z}$).

\noindent
The axiom $\texttt{M}_{11}$, interpreted in $\mathbb{Q}^+$ again, is actually equivalent with the fact that for any sequences $x_1,\cdots,x_q\in\mathbb{Q}^+$ and  $m_1,\cdots,m_q\in\mathbb{N}^+$ none of which divides $n$ (in symbols $m_j\nmid n$), there exists some $y\in\mathbb{Q}^+$ such that $\bigwedge\hspace{-1.5ex}\bigwedge_{j}\neg\Re_{m_j}(y^n\cdot x_j)$. This axiom is not true in $\mathbb{R}^+$ (while $\texttt{M}_{10}$ is true in it) and to see that why $\texttt{M}_{11}$ is true in $\mathbb{Q}^+$ it suffices to note that for given $x_1,\cdots,x_q$ one can take $y$ to be a prime number which does not appear in the unique factorization (of the numerators and denominators of the reduced forms) of any of $x_j$'s. In this case $y^n\cdot x_j$ can be an $m_j$'s power (of a rational number) only when $m_j$ divides $n$. The condition $m_j\nmid n$ is necessary, since otherwise (if $m_j\mid n$ and) if $x_j$ happens to satisfy $\Re_{m_j}(x_j)$, then no   $y$ can satisfy the relation  $\neg\Re_{m_j}(y^n\cdot x_j)$.

$\bullet$
We now show that ${\sf TQ}$ completely axiomatizes the theory of  the structure $\langle\mathbb{Q}^+;<,\times,\square^{-1},{\bf 1},\{\Re_n\}_{n>1}\rangle$ and moreover this structure  admits quantifier elimination, thus the theory of the structure $\langle\mathbb{Q}^+;<,\times\rangle$ is decidable.
For that, we will need the following lemmas.
\begin{lemma}\label{lem-1}\rm
For any $x\in\mathbb{Q}^+$ and any natural   $n_1,n_2>1$,
$$\Re_{n_1}(x)\wedge\Re_{n_2}(x)\iff \Re_n(x),$$
 where $n$ is the least common multiplier of $n_1$ and $n_2$.
\end{lemma}
\begin{proof}
Since $n$ divides $n_1$ and $n_2$, the $\Leftarrow$ part is straightforward; for the  $\Rightarrow$ direction suppose that $x=y^{n_1}=z^{n_2}$. By B\'{e}zout's Identity there are some $c_1,c_2\in\mathbb{Z}$ such that $c_1n/n_1+c_2n/n_2=1$; therefore,
$$x=x^{c_1n/n_1}\cdot x^{c_2n/n_2}=y^{c_1n}\cdot z^{c_2n}=(y^{c_1}z^{c_2})^n,$$
and this completes the proof.
\qed\end{proof}
\begin{lemma}%[The First Quantifier Elimination]
\label{lem-1qe}\rm
For natural numbers $\{n_i\}_{i<p}$ with $n_i>1$ and positive rational numbers $\{t_i\}_{i<p}$ and $x$,
$$\bigwedge\hspace{-2.25ex}\bigwedge_{i<p}\Re_{n_i}(x\cdot t_i) \iff \Re_n(x\cdot\beta)\wedge
\bigwedge\hspace{-2.15ex}\bigwedge_{i\neq j}\Re_{d_{i,j}}(t_i\cdot t_j^{-1}),$$
 where $n$ is the least common multiplier of $n_i$'s, $d_{i,j}$ is the greatest common divisor of $n_i$ and $n_j$ \textup{(}for each $i\neq j$\textup{)} and $\beta=\prod_{i<p}t_i^{c_i(n/n_i)}$ in which $c_i$'s satisfy the (B\'{e}zout's) identity $\sum_{i<p}c_i(n/n_i)=1$.
\end{lemma}
\begin{proof}
For $t_i$'s, $n_i$'s,  $c_i$'s,  $d_{i,j}$'s and $n$  as given above, we show that the relation $\Re_{n_k}(t_k\cdot \beta^{-1})$ holds for each  fixed $k<p$ when $\bigwedge\hspace{-1.5ex}\bigwedge_{i\neq j}\Re_{d_{i,j}}(t_i\cdot t_j^{-1})$ holds.   Let $m_{k,i}$ be  the least common multiplier of $n_k$ and $n_i$ (which is then a divisor of $n$). Let us note that $d_{k,i}/n_i=n_k/m_{k,i}$.
Since $\Re_{d_{k,i}}(t_k\cdot t_i^{-1})$, there should exists some $w_{k,i}$'s (for $i\neq k$) such that $t_k\cdot t_i^{-1}=w_{k,i}^{d_{k,i}}$.
Now, the relation  $\Re_{n_k}(t_k\cdot \beta^{-1})$ follows from the following identities:
\begin{center}
\begin{tabular}{lcl}
$t_k\cdot \beta^{-1}$&$=$&$t_k^{\sum_{i}c_i(n/n_i)}\cdot \prod_{i}t_i^{-c_i(n/n_i)}$\\
&$=$&$\prod_{i\neq k}(t_k\cdot t_i^{-1})^{c_i(n/n_i)}$ \\  &$=$&$\prod_{i\neq k} (w_{k,i}^{d_{k,i}})^{c_i(n/n_i)}$\\
%$$\;=\prod_{i\neq k} (w_{k,i}^{d_{k,i}})^{c_i(n/n_i)}$$
&$=$& $\prod_{i\neq k} w_{k,i}^{c_i\cdot n_k(n/m_{k,i})}$ \\ &$=$& $(\prod_{i\neq k} w_{k,i}^{c_i(n/m_{k,i})})^{n_k}$.
\end{tabular}
\end{center}
\begin{itemize}\itemindent=2.5ex
\item[($\Rightarrow$):]   The relations $\Re_{n_i}(x\cdot t_i)$ and   $\Re_{n_j}(x\cdot t_j)$ immediately imply that  $\Re_{d_{i,j}}(x\cdot t_i)$ and $\Re_{d_{i,j}}(x\cdot t_j)$ and so $\Re_{d_{i,j}}(t_i\cdot t_j^{-1})$. For showing $\Re_n(x\cdot\beta)$ it suffices, by Lemma~\ref{lem-1}, to show that $\Re_{n_i}(x\cdot \beta)$ holds for each $i<p$.   This % immediately
    follows from  $\Re_{n_i}(t_i\cdot \beta^{-1})$, which was proved above, and the assumption $\Re_{n_i}(x\cdot t_i)$.
\item[($\Leftarrow$):]
From  the first part of the proof we have   $\Re_{n_k}(t_k\cdot \beta^{-1})$ for each $k<p$; now by $\Re_n(x\cdot \beta)$ we have $\Re_{n_k}(x\cdot \beta)$ and so $\Re_{n_k}(x\cdot t_k)$  for each $k<p$.
\qed
\end{itemize}
\end{proof}
$\bullet$
Let us note that Lemmas~\ref{lem-1} and~\ref{lem-1qe} are provable in   ${\sf TQ}$. The idea of the proof of Lemma~\ref{lem-1qe} is taken from \cite{ore}.
\begin{lemma}\label{lem-aqe}\rm
The following sentences are provable in ${\sf TQ}$, for any $n>1$\textup{:}

$\forall u\exists y [\Re_n(y\cdot u)]$,

$\forall x,u\exists y [x<y\wedge\Re_n(y\cdot u)]$,

$\forall z,u\exists y [y<z\wedge\Re_n(y\cdot u)]$ and

$\forall x,z,u\exists y [x<z\rightarrow x<y<z\wedge\Re_n(y\cdot u)]$.
\end{lemma}
\begin{proof}
We present a proof for the last formula only. By $\texttt{M}_{10}$ (of Definition~\ref{def-tq}) there exists some $v$ such that $x\cdot u<v^n<z\cdot u$. Then for $y=v^n\cdot u^{-1}$ we will have $x<y<z$ and $\Re_n(y\cdot u)$.
\qed\end{proof}

\begin{lemma}\label{lem-bqe}\rm
The following sentences are provable in ${\sf TQ}$, for any $\{m_j>1\}_{j<q}$\textup{:}

$\forall \{x_j\}_{j<q}\exists y [\bigwedge\hspace{-1.5ex}\bigwedge_{j<q}\neg\Re_{m_j}(y\cdot x_j)]$,

$\forall \{x_j\}_{j<q},u\exists y [u<y\wedge\bigwedge\hspace{-1.5ex}\bigwedge_{j<q}\neg\Re_{m_j}(y\cdot x_j)]$,

$\forall \{x_j\}_{j<q},v\exists y [y<v\wedge \bigwedge\hspace{-1.5ex}\bigwedge_{j<q}\neg\Re_{m_j}(y\cdot x_j)]$ and

$\forall \{x_j\}_{j<q},u,v\exists y [u<v\rightarrow u<y<v\wedge \bigwedge\hspace{-1.5ex}\bigwedge_{j<q}\neg\Re_{m_j}(y\cdot x_j)]$.
\end{lemma}
\begin{proof}
The first sentence is an immediate consequence of $\texttt{M}_{11}$ (of Definition~\ref{def-tq}) for $n=1$. We show the last sentence. There exists  $\gamma$, by $\texttt{M}_{11}$, such that the relation $\bigwedge\hspace{-1.5ex}\bigwedge_{j}\neg\Re_{m_j}(\gamma\cdot x_j)$ holds. Let $M=\prod_jm_j$; by $\texttt{M}_{10}$ there exists some $\delta$ such that the inequalities $u\cdot\gamma^{-1}<\delta^M<v\cdot\gamma^{-1}$ holds. Now for $y=\gamma\cdot\delta^M$ we have $u<y<v$ and also $\bigwedge\hspace{-1.5ex}\bigwedge_{j}\neg\Re_{m_j}(y\cdot x_j)$, since if (otherwise) we had $\Re_{m_j}(y\cdot x_j)$, then $\Re_{m_j}(\gamma\cdot\delta^M\cdot x_j)$ and so $\Re_{m_j}(\gamma\cdot x_j)$ would hold; a contradiction.
\qed\end{proof}
\begin{lemma}\label{lem-cqe}\rm
In the theory ${\sf TQ}$ the following formulas

$\exists x [\Re_n(x\cdot t)\wedge \bigwedge\hspace{-1.5ex}\bigwedge_{j<q}\neg\Re_{m_j}(x\cdot s_j)]$,

$\exists x [u<x\wedge \Re_n(x\cdot t)\wedge \bigwedge\hspace{-1.5ex}\bigwedge_{j<q}\neg\Re_{m_j}(x\cdot s_j)]$ and

$\exists x [x<v\wedge \Re_n(x\cdot t)\wedge \bigwedge\hspace{-1.5ex}\bigwedge_{j<q}\neg\Re_{m_j}(x\cdot s_j)]$

\noindent
are equivalent with

\qquad $\bigwedge\hspace{-1.5ex}\bigwedge_{m_j\mid n (j<q)}\neg\Re_{m_j}(t^{-1}\cdot s_j)$\textup{;}

\noindent
and the formula

$\exists x [u<x<v\wedge \Re_n(x\cdot t)\wedge \bigwedge\hspace{-1.5ex}\bigwedge_{j<q}\neg\Re_{m_j}(x\cdot s_j)]$

\noindent
is equivalent with

\qquad $\bigwedge\hspace{-1.5ex}\bigwedge_{m_j\mid n (j<q)}\neg\Re_{m_j}(t^{-1}\cdot s_j)\wedge u<v$.
\end{lemma}
\begin{proof}
If $m_j\mid n$ then $\Re_n(x\cdot t)$ implies $\Re_{m_j}(x\cdot t)$. Now, if $\Re_{m_j}(t^{-1}\cdot s_j)$ were true, then $\Re_{m_j}(x\cdot s_j)$ would be true too;  contradicting $\bigwedge\hspace{-1.5ex}\bigwedge_{j<q}\neg\Re_{m_j}(x\cdot s_j)$. Suppose now that the relation $\bigwedge\hspace{-1.5ex}\bigwedge_{m_j\mid n}\neg\Re_{m_j}(t^{-1}\cdot s_j)$ holds. By $\texttt{M}_{11}$ there exists some $\gamma$ such that $\bigwedge\hspace{-1.5ex}\bigwedge_{m_j\nmid n}\neg\Re_{m_j}(\gamma\cdot t^{-1}\cdot s_j)$ holds. By $\texttt{M}_{10}$ there exists some $\delta$ such that the inequalities $u\cdot t\cdot\gamma^{-n}<\delta^{M\cdot n}<v\cdot t\cdot\gamma^{-n}$ (if $u<v$) hold,  where $M$ is the product $\prod_{j<q}m_j$. For $x=\delta^{M\cdot n}\cdot \gamma^{n}\cdot t^{-1}$ we have $u<x<v$ and $\Re_n(x\cdot t)$. We show $\neg\Re_{m_j}(x\cdot s_j)$ for each $j<q$ by distinguishing two cases: if $m_j\mid n$ then $\neg\Re_{m_j}(t^{-1}\cdot s_j)$ implies the relation $\neg\Re_{m_j}(\delta^{M\cdot n}\cdot \gamma^{n}\cdot t^{-1}\cdot s_j)$; if $m_j\nmid n$ then by  $\neg\Re_{m_j}(\gamma\cdot t^{-1}\cdot s_j)$ we have the relation  $\neg\Re_{m_j}(\delta^{M\cdot n}\cdot \gamma^{n}\cdot t^{-1}\cdot s_j)$.
\qed\end{proof}

$\bullet$
Finally we can prove the main result which appears for the first time in this thesis.
\begin{theorem}\rm\label{omq+}
The infinite theory ${\sf TQ}$ completely axiomatizes the theory of the structure   $\langle\mathbb{Q}^+;<,\times\rangle$, and moreover the structure    $\langle\mathbb{Q}^+;<,\times,\square^{-1},{\bf 1},\{\Re_n\}_{n>1}\rangle$   admits quantifier elimination.

\end{theorem}
\begin{proof}
We are to eliminate the quantifier of the formula
\begin{equation}\label{q-1}
\exists x (\bigwedge\hspace{-2.2ex}\bigwedge_{i<p} \Re_{n_i}(x^{a_i}\cdot t_i) \;\wedge\; \bigwedge\hspace{-2.25ex}\bigwedge_{j<q} \neg\Re_{m_j}(x^{b_j}\cdot s_j) \;\wedge\;
  \bigwedge\hspace{-2.35ex}\bigwedge_{k<f} u_k\!<\!x^{c_k} \;\wedge\;  \bigwedge\hspace{-2.25ex}\bigwedge_{\ell<g} x^{d_\ell}\!<\!v_\ell \;\wedge\;
 \bigwedge\hspace{-2.25ex}\bigwedge_{\iota<h} x^{e_\iota}=w_\iota).
\end{equation}
By the equivalences
\begin{itemize}\itemindent=10em
\item[(i)]
$a^n<b^n \leftrightarrow a<b$
\item[(ii)]
$\Re_{m\cdot n}(a^n)\leftrightarrow \Re_m(a)$
\end{itemize}
we can assume that all the $a_i$'s, $b_j$'s, $c_k$'s, $d_\ell$'s and $e_\iota$'s are equal to each other, and moreover, equal to one (cf. the proof of Theorem~\ref{thm-oaz}). We can also assume that $h=0$ and that $f,g\leqslant 1$. By Lemma~\ref{lem-1qe}   we can also assume that $p\leqslant 1$.
\begin{itemize}
\item[--]
If $q=0$, then Lemma~\ref{lem-aqe} implies that the quantifier of the  formula~\eqref{q-1} can be eliminated. So, we assume that $q>0$.
\item[--]
If $p=0$, then the quantifier of \eqref{q-1} can be eliminated by Lemma~\ref{lem-bqe}.
\item[--]
Finally, if $p=1$ (and $q\neq 0=h$ and  $f,g\leqslant 1$), then Lemma~\ref{lem-cqe} implies that the formula~\eqref{q-1} is equivalent with a quantifier-free formula.
\end{itemize}\vspace{-1cm}
\qed\end{proof}
\begin{cor}\label{omq}\rm
The below infinite theory completely axiomatized the theory of the structure $\langle\mathbb{Q};<,\times\rangle$:
\begin{equation*}
\begin{tabular}{l}
($\texttt{O}_1$) \; $\forall x,y(x<y\rightarrow y\not<x)$  \\
($\texttt{O}_2$) \; $\forall x,y,z (x<y<z\rightarrow x<z)$ \\
($\texttt{O}_3$) \; $\forall x,y (x<y \vee x=y \vee y<x)$ \\
($\texttt{M}_1$) \; $\forall x,y,z\,(x\cdot(y\cdot z)=(x\cdot y)\cdot z)$ \\
($\texttt{M}_2^\circ$) \; $\forall x (x\cdot\mathbf{1}=x  \;\,  \wedge \;\, x\cdot{\bf 0}={\bf 0}={\bf 0}^{-1})$ \\
($\texttt{M}_3^\circ$) \; $\forall x (x\neq {\bf 0}\rightarrow x\cdot  x^{-1}=\mathbf{1})$ \\
($\texttt{M}_4$) \; $\forall x,y (x\cdot y=y\cdot x)$ \\
($\texttt{M}_5^\circ$) \; $\forall x,y,z(x<y\wedge {\bf 0}<z \rightarrow x\cdot z<y\cdot z)$ \\
($\texttt{M}_5^\bullet$) \; $\forall x,y,z(x<y\wedge z<{\bf 0} \rightarrow y\cdot z<x\cdot z)$  \\
($\texttt{M}_6^\circ$) \; $\exists y ({\bf -1}<{\bf 0}<{\bf 1}<y)$ \\
($\texttt{M}_{8}$) \;  $\forall
x (x^{2n}={\bf 1}\longleftrightarrow x={\bf 1}\vee x={\bf -1})$ \\
$(\texttt{M}_{10}^\circ)$    $\forall x,z\exists y ({\bf 0}<x<z\rightarrow x<y^n<z)\qquad\qquad n\in\mathbb{N}$\\
 ($\texttt{M}_{11}$) \;  $\forall \{x_j\}_{j<q} \exists y \forall z  \bigwedge\hspace{-1.5ex}\bigwedge_{m_j\nmid n (j<q)} (y^n\cdot x_j \neq z^{m_j})$ \quad  for each $n\geqslant 1$ (and $m_j>1$)
 \end{tabular}
\end{equation*}
and moreover the structure $\langle\mathbb{Q};<,\times,\square^{-1},{\bf -1},{\bf 0},{\bf 1},\{\Re_n\}_{n>1}\rangle$   admits quantifier elimination.
\end{cor}
\begin{proof}
Quantifier elimination of the theory of %the structure
$\langle\mathbb{Q};<,\times,\square^{-1},{\bf -1},{\bf 0},{\bf 1},\{\Re_n\}_{n>1}\rangle$ follows from Theorem~\ref{omq+}: it suffices to distinguish  the signs by noting that for all $x$ one of the three cases $-x>{\bf 0}$ or $x={\bf 0}$ or $x>{\bf 0}$ holds.
\qed\end{proof}
\begin{pro}\label{qam}\rm\label{subsec-q}
The theory of the structure $\langle\mathbb{Q};+,\times\rangle$ is undecidable.
\end{pro}
\begin{proof}
Since the set of integer numbers is definable in $\langle\mathbb{Q};+,\times\rangle$ \cite{robinson},
%$$\langle\mathbb{Z};+,\times\rangle\preccurlyeq\langle\mathbb{Q};+,\times\rangle$$
 the decidability of the theory of the structure $\langle\mathbb{Q};+,\times\rangle$ implies the decidability of the theory of the structure $\langle\mathbb{Z};+,\times\rangle$ and this contradicts Proposition~\ref{zam}.
\qed\end{proof}

\subsection{Non-finite Axiomatizability of $\langle\mathbb{Q};<,\times\rangle$}
\begin{theorem}\rm
The structure $\langle\mathbb{Q}^+;<,\times\rangle$ is not finitely axiomatizable.
\end{theorem}
\begin{proof}
To see that the structure $\langle\mathbb{Q}^+;<,\times\rangle$ cannot be axiomatized by a finite set of sentences we present an ordered multiplicative structure that satisfies any sufficiently large  finite number of the axioms of ${\sf TQ}$ but does not satisfy all of its axioms. Let $\mathfrak{p}$ be a sufficiently large prime number. The set
$$\mathbb{Q}/\mathfrak{p}=\{m/\mathfrak{p}^k\mid m\in\mathbb{Z},k\in\mathbb{N}\}$$
is closed under addition and the operation $x\mapsto x/\mathfrak{p}$,  and the inclusions  $\mathbb{Z}\subset\mathbb{Q}/\mathfrak{p}\subset\mathbb{Q}$ hold.
Let $\rho_0,\rho_1,\rho_2,\cdots$ denote  the sequence of all prime numbers ($2,3,5,\cdots$). Let $(\mathbb{Q}/\mathfrak{p})^\ast$ be the set $\{\prod_{i<\ell}\rho_i^{r_i}\mid \ell\in\mathbb{N},r_i\in\mathbb{Q}/\mathfrak{p}\}$;  this  is closed under multiplication and the operation $x\mapsto {x}^{1/\mathfrak{p}}$, and we have the inclusions $\mathbb{Q}^+\subset (\mathbb{Q}/\mathfrak{p})^\ast \subset\mathbb{R}^+$. Thus, $(\mathbb{Q}/\mathfrak{p})^\ast$  satisfies the axioms $\texttt{O}_{1}$, $\texttt{O}_{2}$, $\texttt{O}_{3}$, $\texttt{M}_{1}$, $\texttt{M}_{2}$, $\texttt{M}_{3}$, $\texttt{M}_{4}$, $\texttt{M}_{5}$  and $\texttt{M}_{6}$ of Proposition~\ref{thm-or+}, and also the axiom $\texttt{M}_{10}$. However, it does not satisfy the axiom $\texttt{M}_{11}$ for $n=q=x_0=1$ and $m_0=\mathfrak{p}$ because $(\mathbb{Q}/\mathfrak{p})^\ast\models\forall y\Re_{\mathfrak{p}}(y)$. We show that $(\mathbb{Q}/\mathfrak{p})^\ast$ satisfies the instances of the axiom $\texttt{M}_{11}$ when $1<m_j<\mathfrak{p}$ (for each $j<q$  and arbitrary $n,q$). Thus, no finite number of the instances of $\texttt{M}_{11}$ can prove all of its instances (with the rest of the axioms of ${\sf TQ}$). Let $x_j$'s be given from $(\mathbb{Q}/\mathfrak{p})^\ast$; write $x_j=\prod_{i<\ell_j}\rho_i^{r_{i,j}}$ where we can assume that $\ell_j\geqslant q$. Put $r_{j,j}=u_j/\mathfrak{p}^{v_j}$ where $u_j\in\mathbb{Z}$ and $v_j\in\mathbb{N}$ (for each $j<q$).
Define $t_j$ to be $1$ when $m_j\mid u_j$ and be $m_j$ when $m_j\nmid u_j$. Let $$y=\prod_{i<q}\rho_i^{(t_i/\mathfrak{p}^{v_i+1})} (\in(\mathbb{Q}/\mathfrak{p})^\ast).$$
We show
$$\bigwedge\hspace{-2.4ex}\bigwedge_{j<q}\neg\Re_{m_j}(y^n\cdot x_j)$$
under the assumption $\bigwedge\hspace{-1.5ex}\bigwedge_{j<q} m_j\nmid n$. Take a $k<q$, and assume (for the sake of contradiction) that $\Re_{m_k}(y^n\cdot x_k)$. Then $\Re_{m_k}(\rho_k^{nt_k/\mathfrak{p}^{v_k+1}}\cdot
\rho_k^{u_k/\mathfrak{p}^{v_k}})$ holds, and so there should exist some $a,b$ %with $(a,\mathfrak{p})=1$
such that
$$\rho_k^{(nt_k+\mathfrak{p}u_k)/\mathfrak{p}^{v_k+1}}=
\rho_k^{(m_k\cdot a)/\mathfrak{p}^b}.$$
Therefore,
$$m_k\mid nt_k+\mathfrak{p}u_k.$$
We reach to a contradiction by distinguishing two cases:

\noindent
(i) if $m_k\mid u_k$ then $t_k=1$ and so $m_k\mid n+\mathfrak{p}u_k$ whence $m_k\mid n$,  contradicting $\bigwedge\hspace{-1.5ex}\bigwedge_{j<q} m_j\nmid n$;

\noindent
(ii) if $m_k\nmid u_k$ then $t_k=m_k$ and so $m_k\mid nm_k+\mathfrak{p}u_k$ whence $m_k\mid \mathfrak{p}u_k$ which by $(m_k,\mathfrak{p})=1$ implies that $m_k\mid u_k$, contradicting the assumption (of $m_k\nmid u_k$).
\qed\end{proof} 
% Chapter Template

\chapter{Conclusions and Open Problems} % Main chapter title
\label{Chapter5} % Change X to a consecutive number; for referencing this chapter elsewhere, use \ref{ChapterX}

%----------------------------------------------------------------------------------------
%	SECTION 1
%----------------------------------------------------------------------------------------

\section{Some Conclusions}
In the following table the decidable structures are denoted by $\Delta_1$\label{delta1} and the undecidable ones by $\Delta_1\hspace{-4.45mm}\backslash\hspace{-3.5mm}\not$\label{ndelta1}\quad:

\vspace{1cm}
\begin{center}
%\begin{LTR}
\begin{tabular}{|c||c|c|c|c|}
\hline
  & $\mathbb{N}$ & $\mathbb{Z}$ & $\mathbb{Q}$ & $\mathbb{R}$ \\
\hline
\hline
$\{<\}$ & $\Delta_1$ & $\Delta_1$ & $\Delta_1$ & $\Delta_1$ \\
\hline
$\{<,+\}$ & $\Delta_1$ & $\Delta_1$ & $\Delta_1$ & $\Delta_1$ \\
\hline
$\{<,\times\}$ & $\hspace{-2.55mm}\Delta_1\hspace{-4.45mm}\backslash\hspace{-3.5mm}\not$ \;  & $\hspace{-2.55mm}\Delta_1\hspace{-4.45mm}\backslash\hspace{-3.5mm}\not$ \;  &   $\Delta_1$     & $\Delta_1$ \\
\hline
\hline
$\{+,\times\}$ & $\hspace{-2.55mm}\Delta_1\hspace{-4.45mm}\backslash\hspace{-3.5mm}\not$ \;  & $\hspace{-2.55mm}\Delta_1\hspace{-4.45mm}\backslash\hspace{-3.5mm}\not$ \;  &  \ $\hspace{-2.55mm}\Delta_1\hspace{-4.45mm}\backslash\hspace{-3.5mm}\not$ \;   & $\Delta_1$ \\
\hline
\end{tabular}
%\end{LTR}
\end{center}
\vspace{1cm}
\begin{itemize}
\item[$\bullet$]
Decidability of the theory of the structure $\langle\mathbb{Q};<,\times\rangle$ and also the presentation of  an explicit axiomatization for the theory of the structure $\langle\mathbb{R};<,\times\rangle$ are some new results in this thesis.
\item[$\bullet$]
For the  theory of some other decidable structures, the old and new (syntactic) proofs were given along with some explicit axiomatizations.
\item[$\bullet$]  It is interesting to note that
\begin{itemize}
\item  the undecidability of the theories of %the structures
$\langle\mathbb{N};<,\times\rangle$ and $\langle\mathbb{Z};<,\times\rangle$ follow from the undecidability of the theories of $\langle\mathbb{N};+,\times\rangle$
 and $\langle\mathbb{Z};+,\times\rangle$ (and the definability of $+$ from $<$ and $\times$ in $\mathbb{N}$ and $\mathbb{Z}$);
\item the decidability of the theory of the structure $\langle\mathbb{R};<,\times\rangle$ follows from the decidability of the theory of the structure $\langle\mathbb{R};+,\times\rangle$ (and the definability of $<$ from $+$ and $\times$ in $\mathbb{R}$);
\item  though, the undecidability of the additive and multiplicative structure $\langle\mathbb{Q};+,\times\rangle$ has nothing to do with the (decidable) theory of  multiplicative structure $\langle\mathbb{Q};<,\times\rangle$; as a matter of fact $+$ is not definable in the multiplicative structure $\langle\mathbb{Q};<,\times\rangle$ while $<$ is definable in $\langle\mathbb{Q};+,\times\rangle$.
\end{itemize}
\end{itemize}

%----------------------------------------------------------------------------------------
%	SECTION 2
%----------------------------------------------------------------------------------------

\section{Some Open Problems}

There are lots of notable sets between $\mathbb{Q}$ and $\mathbb{R}$. For example
\begin{itemize}
\item[--]
$\mathbb{Q}[\sqrt{2}].$
\item[--]
$\mathbb{Q}[\sqrt{2},\sqrt{3},\sqrt{5},\cdots].$
\item[--] ${\vspace{1.4cm}^{_{|}}\hspace{-0.93ex}{\Omega}\hspace{-0.77ex}^{_{|}}}=$\label{omega} the set of real numbers that are constructible by ruler and compass.
\item[--] The field generated by the radicals of rational numbers (when they exist in the real numbers).
\end{itemize}

For any set $A$ with $\mathbb{Q}\subseteq A\subseteq\mathbb{R}$, Theorem~\ref{thm-oa} axiomatizes the theory of the structure $\langle A;<,+\rangle$ when $A$ is closed under the addition operation  and also the operations $x\mapsto x/n\ \ (n\in\mathbb{N^{+}})$. But the theory of the structure $\langle A;<,\times\rangle$ could be different, when $A$ is closed under $\times$ (it could not be even axiomatizable, or be axiomatizable by a different set of axioms). For example, it is not yet known if the theory of the structure $\langle\,{\vspace{1.4cm}^{_{|}}\hspace{-0.93ex}{\Omega}\hspace{-0.77ex}^{_{|}}}\,;
<,\times\rangle$ is decidable or not!?

Investigating any of these problems could lead to some wonderful results in Mathematical Logic and Computer Science.

\chapter*{Index}
\addcontentsline{toc}{chapter}{Index}
\markboth{}{Index}
\section*{A}
Abelian Group
\quad
\pageref{ag}
\\
Axiomatizability
\quad
\pageref{ax}
\section*{B}
B\'{e}zout's Theorem
\quad
\pageref{bz}
\section*{C}
Chinese Remainder Theorem
\quad
\pageref{cr}
\\
Complete Theory
\quad
\pageref{cty}
\section*{D}
Decidable Set
\quad
\pageref{ds}
\\
Decision Algorithm
\quad
\pageref{da}
\\
Dense Linear Order
\quad
\pageref{dlo}
\\
Discrete Order
\quad
\pageref{dco}
\\
Disjunctive Normal Form
\quad
\pageref{djnf}
\\
Divisible Group
\quad
\pageref{dvg}
\section*{E}
Effectively Enumerable Set
\quad
\pageref{ees}
\\
Entscheidungsproblem
\quad
\pageref{dp}
\section*{F}
Finitely Axiomatizable
\quad
\pageref{fa}
\section*{G}
Generalized Chinese Remainder Theorem
\quad
\pageref{crt}
\\
Group
\quad
\pageref{g}
\section*{L}
Lagrange's Four Square Theorem
\quad
\pageref{fslt}
\section*{M}
Main Lemma of Quantifier Elimination
\quad
\pageref{mainlem}
\\
\section*{N}
Non-trivial Group
\quad
\pageref{ntg}
\section*{O}
Ordered Group
\quad
\pageref{og}
\\
Ordered Structure
\quad
\pageref{def-os}
\\
Orders Without Endpoints
\quad
\pageref{weo}
\section*{Q}
Quantifier Elimination
\quad
\pageref{e}
\section*{S}
Successor
\quad
\pageref{su}
\section*{T}
Tarski-Robinson's Identity
\quad
\pageref{r}
\\
Tarski-Seidenberg's Theorem
\quad
\pageref{tst}
\\
Theory
\qquad
\pageref{ty}

\section*{List of Symbols}
%\addcontentsline{toc}{chapter}{List of Symbols}
%\markboth{}{List of Symbols}
\begin{tabular}{ll}
  $\Delta_1$ & $\pageref{delta1}$ \\
  $\Delta_1\hspace{-4.45mm}\backslash\hspace{-3.5mm}\not$ & $\pageref{ndelta1}$ \\
  ${\vspace{1.4cm}^{_{|}}\hspace{-0.93ex}{\Omega}\hspace{-0.77ex}^{_{|}}}$ & $\pageref{omega}$ \\
  $\Re_n(y)$ & $\pageref{rn}$ \\
  ${\sf TQ}$ & $\pageref{tq}$
\end{tabular}

%----------------------------------------------------------------------------------------
%	THESIS CONTENT - APPENDICES
%----------------------------------------------------------------------------------------

%\appendix % Cue to tell LaTeX that the following "chapters" are Appendices

% Include the appendices of the thesis as separate files from the Appendices folder
% Uncomment the lines as you write the Appendices

%\include{Appendices/AppendixA}
%\include{Appendices/AppendixB}
%\include{Appendices/AppendixC}

%----------------------------------------------------------------------------------------
%	BIBLIOGRAPHY
%----------------------------------------------------------------------------------------

%\printbibliography[heading=bibintoc]

%\begin{thebibliography}{24}
%
%\bibitem{az}
%\newblock{\sc Zofia Adamowicz \& Pawel Zbierski},
% \newblock{\bf Logic of Mathematics: a modern course of classical logic},
% \newblock Wiley (1997), {\sc isbn}:~{9780471060260}.
%
%\end{thebibliography}
%----------------------------------------------------------------------------------------


\begin{thebibliography}{24}
\addcontentsline{toc}{chapter}{Bibliography}
%\rm\begin{latin}
%\Roman

%\baselineskip = 4.5mm

%\begin{thebibliography}{}


\bibitem {az}
\newblock {\sc Zofia Adamowicz \& Pawel Zbierski}, \newblock {\bf Logic of Mathematics: A Modern Course of Classical Logic}, \newblock John Wiley \& Sons (1997), {\sc isbn}:~{9780471060260}.




\bibitem{assadi-salehi}
\newblock {\sc Ziba Assadi \& Saeed Salehi},
\newblock  {\em On Decidability and Axiomatizability of Some Ordered Structures},
\newblock {\bf Soft Computing}
\newblock 23:11 (2019) 3615--3626.
\newblock  {\sc doi}:~{10.1007/s00500-018-3247-1}.

\bibitem{bbj}
\newblock {\sc
George S. Boolos \& John P. Burgess \& Richard C. Jeffrey},
\newblock {\bf Computability and Logic},
\newblock Cambridge University Press (5th ed. 2007), {\sc isbn}:~{9780521701464}.



%%%%%%%%%%%%%%%%%%%%%%
\bibitem{bcr}
\newblock {\sc Jacek Bochnak \& Michel Coste \& Marie--Fran\c{c}oise Roy},
\newblock {\bf Real Algebraic Geometry},
\newblock Springer (1998), {\sc isbn}:~{9783642084294}.




\bibitem{bpr}
\newblock {\sc Saugata Basu \&  Richard Pollack \&
Marie--Fran\c{c}oise  Roy},
\newblock {\bf Algorithms in Real Algebraic Geometry},
\newblock Springer (2006),  {\sc isbn}:~{9783540330981}.
%%%%%%%%%%%%%%%%%%%%%%



\bibitem{cegielski}
\newblock {\sc Patrick C\'{e}gielski},
\newblock ``{\em Th\'eorie \'El\'ementaire de la Multiplication des Entiers Naturels}'',
\newblock in:  C.~Berline, K.~McAloon, J.-P.~Ressayre (eds.),
\newblock {\bf Model Theory and Arithmetic},
\newblock Comptes Rendus
 d'une Action Th\'{e}matique Programm\'{e}e du C.N.R.S. sur la Th\'{e}orie des  Mod\`{e}les et l'Arithm\'{e}tique,
 Paris, France, 1979/80,
\newblock Lecture Notes in Mathematics 890,  Springer (1981),
\newblock {\sc isbn}:~9783540111597,  pp.~44--89.
\newblock {\sc doi}:~10.1007/BFb0095657.



\bibitem{enderton}
\newblock  {\sc  Herbert B. Enderton},
\newblock   {\bf A Mathematical Introduction to Logic}, \newblock  Academic Press    (2nd ed. 2001),  {\sc isbn}:~{9780122384523}.




\bibitem{frankel}
\newblock {\sc  Aviezri S. Fraenkel},
\newblock  {\em New Proof of the Generalized Chinese Remainder Theorem},
\newblock   {\bf  The Proceedings of the American Mathematical  Society} 14:5      (1963) 790--791.
\newblock  {\sc doi}:~{10.1090/S0002-9939-1963-0154841-6}.




\bibitem{hinman}
\newblock  {\sc  Peter G. Hinman},
\newblock   {\bf Fundamentals of Mathematical Logic},
\newblock  CRC Press    (2005),
{\sc isbn}:~{9781568812625}.




\bibitem{kk}
\newblock {\sc Georg Kreisel \& Jean Louis Krivine},
\newblock {\bf Elements of Mathematical Logic: Model Theory}, \newblock North--Holland (1971),
{\sc isbn}:~{9780720422658}.



\bibitem{marker}
\newblock  {\sc  David Marker},
\newblock   {\bf Model Theory: An Introduction},
\newblock  Springer    (2002),  {\sc isbn}:~{9781441931573}.



\bibitem{monk}
\newblock  {\sc  J. Donald Monk},
\newblock   {\bf Mathematical Logic},
\newblock  Springer    (1976),
{\sc isbn}:~{9780387901701}.



\bibitem{mostowski}
\newblock {\sc Andrzej Mostowski},
\newblock {\em On Direct Products of Theories},
\newblock {\bf The Journal of Symbolic Logic} 17 (1952) 1--31.
\newblock {\sc doi}:~{10.2307/2267454}.



\bibitem{ore}
\newblock {\sc  Oystein Ore},
\newblock  {\em The General Chinese Remainder Theorem},
\newblock   {\bf  The American Mathematical Monthly} 59:6
     (1952) 365--370.
\newblock  {\sc doi}:~{10.2307/2306804}.



\bibitem{rz}
\newblock {\sc Abraham Robinson \& Elias Zakon},
\newblock {\em Elementary Properties of Ordered Abelian Groups},
\newblock {\bf Transactions of the American Mathematical Society} 96:2 (1960) 222–236.
\newblock {\sc doi}:~{10.2307/199346}.



\bibitem{robinson}
\newblock {\sc  Julia Robinson},
\newblock  {\em Definability and Decision Problems in Arithmetic},
\newblock   {\bf  The Journal of Symbolic Logic} 14:2
     (1949) 98--114.
\newblock  {\sc doi}:~{10.2307/2266510}.



\bibitem{salehi-m}
\newblock {\sc  Saeed Salehi},
\newblock ``{\em Axiomatizing Mathematical Theories: Multiplication}'',
\newblock in: A. Kamali-Nejad (ed.),
\newblock {\bf Proceedings of Frontiers in Mathematical Sciences},
\newblock  Sharif University
of Technology,  Tehran, Iran (2012), pp.~165--176.
\newblock {https://arxiv.org/pdf/1612.06525.pdf}


\bibitem{salehi-mo}
\newblock {\sc  Saeed Salehi},
\newblock ``{\em Computation in Logic and Logic in Computation}'',
\newblock in: B. Sadeghi-Bigham (ed.),
\newblock {\bf Proceedings of the Third International Conference on Contemporary Issues in Computer and Information Sciences (CICIS 2012)},
\newblock  Brown Walker Press, USA (2012), pp.~580--583.
\newblock {https://arxiv.org/pdf/1612.06526.pdf}





\bibitem{smorynski}
\newblock  {\sc  Craig Smory\'nski},
\newblock   {\bf Logical Number Theory I: An Introduction}, \newblock  Springer    (1991),  {\sc isbn}:~{9783540522362}.



\bibitem{visser}
\newblock {\sc  Albert Visser},
\newblock  {\em On \textsf{Q}},
\newblock   {\bf  Soft Computing} 21:1
     (2017) 39--56. {\sc doi}:~{10.1007/s00500-016-2341-5}.





\end{thebibliography}
\end{document}